\documentclass[11pt,a4paper,twoside]{article}
\usepackage[utf8]{inputenc}
  
            \usepackage[english]{babel}
            \usepackage[leqno]{amsmath} 
            \usepackage{amssymb,amsthm}
            \usepackage{mathtools}
            \usepackage{bbm}
            \usepackage[twoside, bindingoffset=5mm, left=15mm, right=25mm, top=28mm, bottom=28mm,a4paper]{geometry} 
            \usepackage{float}
            \usepackage{fancyhdr}
            \usepackage[breaklinks,plainpages=false,pdfpagelabels=true]{hyperref}
            \usepackage{bookmark}
            \usepackage{enumerate}
            \usepackage{newpxtext}
            \usepackage{keyval}
            \usepackage{constants}
            \usepackage{sectsty}
            \usepackage{graphicx}
            \usepackage{tikz}
            \usetikzlibrary{arrows,calc,shapes,scopes,matrix,intersections}
            \usetikzlibrary{decorations,decorations.markings,decorations.pathreplacing,decorations.pathmorphing}
            \usepackage{pgfplots}
            \pgfplotsset{compat=1.18}
            \usepackage{comment}
            \newcommand{\Beta}{\operatorname{B}}
            \newcommand{\R}{\mathbbm {R}}
            
            \newcommand{\N}{\mathbbm {N}}
            
            \newcommand{\W}{\mathcal {W}}
            
            \newcommand{\diff}{\operatorname{d}\!}

            \newcommand{\Leb}{\mathcal {L}}

            \newcommand{\mrest}{%
              \mathbin{\vrule height 1.6ex depth 0pt width
            0.13ex\vrule height 0.13ex depth 0pt width 1.3ex}
            }
            
            \newcommand {\cb}[1] {\tikz[baseline=(char.base)]{
                        \node[shape=circle,draw,inner sep=2pt] (char) {\itshape{\textbf{#1}}};}}
            \newcommand {\cbm}[1] {\tikz[baseline=(char.base)]{
                        \node[shape=circle,draw,inner sep=2pt] (char) {\scriptsize\itshape{\textbf{#1}}};}}

            \newtheoremstyle{plain}	
            	{1em}{1em}	
            	{\itshape}	
            	{}	
            	{\bfseries}	
            	{}{\newline}	
            	{\thmnumber{#2 }\thmname{#1}\thmnote{ (#3)}}

            \newtheorem{thm}{Theorem}
            \newtheorem{corollary}[thm]{Corollary}
            \newtheorem{lemma}[thm]{Lemma}

            \newtheoremstyle{plainupright}
            	{1em}{1em}	
            	{\normalfont}	
            	{}	
            	{\bfseries}	
            	{}{\newline}	
{\thmnumber{#2 }\thmname{#1}\thmnote{ (#3)}}

\theoremstyle{plainupright}
\newtheorem{remark}[thm]{Remark}

\newlength{\paperinfolabelwidth}

\newcommand{\paperinfo}[2]{%
  \par\smallskip
  \begingroup
    \footnotesize
    \settowidth{\paperinfolabelwidth}{\textit{#1:}\enspace}%
    \hangindent=\paperinfolabelwidth
    \hangafter=1
    \noindent
    \textit{#1:}\enspace #2\par
  \endgroup
}

\providecommand{\keywords}[1]{%
  \paperinfo{Keywords}{#1}%
}

\providecommand{\pacs}[2][MSC 2020]{%
  \paperinfo{#1}{#2}%
}

            \title{Regular Curves, Singular Graphs: Cantor Parts and the Relaxed Willmore Energy}
            \author{Hans-Christoph Grunau, Boris Gulyak\\
            \small  Institut für Analysis und Numerik,\\
                \small  Otto-von-Guericke-Universität,\\
               \small  Fakultät für Mathematik,\\
               \small   Postfach 4120,\\
               \small  39016 Magdeburg,\\
               \small Germany\\
                \small \texttt{bgulyak@ovgu.de}}

\begin{document}

            \maketitle

\begin{abstract}

One might expect that finite relaxed elastic energy rules out diffuse singularities in the derivative, leaving only absolutely continuous and jump parts. This is suggested by the role of $SBV$ in free-discontinuity problems and by interpreting jumps as vertical segments of limiting graphs. We show that it fails for the relaxed one-dimensional Willmore energy. We construct a continuous function $u\in BV((0,1))$ with $D^c u\neq0$ and $\overline{\W}(u)<\infty $, so finite relaxed Willmore energy does not imply $u\in SBV((0,1))$. The idea is to concentrate the Cantor part exactly where the absolutely continuous slope blows up. There the singular diffuse measure meets the blow-up condition of the relaxation theorem, while the weighted curvature term stays integrable.

Geometrically, the example shows that Cantor parts of $BV$-graph derivatives need not be intrinsic singularities of the underlying curve. The graph has an arc-length parametrization of class $C^1\cap W^{2,2}$, and a suitable rotation turns it into a Lipschitz graph whose derivative has no singular part. The construction also rescales to make the relaxed Willmore energy arbitrarily small, and it extends to relaxed $L^p$-curvature energies for all $p>1$.


\end{abstract}

\bigskip         
            
\keywords{relaxed Willmore energy,  elastic energy of curves, $L^p$-curvature energies, functions of bounded variation,  Cantor part, lower semicontinuous relaxation}

\pacs[MSC Classification]{Primary 49J45; Secondary 26A45, 49Q20, 53A04, 74K10}

\section{Introduction}

Curvature energies for graphs form a simple but rich class of geometric variational problems. If $u\in W^{2,p}((0,1))$, the $p$-curvature energy  for $p>1$ of its graph is  
\[   \W_p(u) :=  \int_{\operatorname{graph}[u]}  |\kappa_u|^p \diff s_u
    =  \int_0^1  \frac{|u''(x)|^p}  {\bigl(1+u'(x)^2\bigr)^{(3p-1)/2}}   \diff x \]
where
\[  \kappa_u(x)
    =\frac{\diff\ }{\diff x} \left(\frac{u'(x)}{\sqrt{1+u'(x)^2}} \right)
    = \frac{u''(x)}{\bigl(1+u'(x)^2\bigr)^{3/2}},
    \qquad \diff s_u
    =     \sqrt{1+u'(x)^2} \diff x. \]
Thus the one-dimensional Willmore, also called elastic energy, is $\W(u)=\W_2(u)$. 
For $p>1$ and $u\in L^1((0,1))$, we define the lower semicontinuous
envelope of $\W_p$ with respect to strong $L^1$-convergence by
\[    \overline{\W}_p(u)
    := \inf \left\{ \liminf_{j\to\infty}\W_p(u_j) \,\middle|\, u_j\in W^{2,p}((0,1)),\,   u_j\to u \text{ in }L^1(0,1)   \right\} .\]
We refer to $\overline{\W}_p$ as the relaxed $p$-elastic energy.
 In particular, $\overline{\W}:=\overline{\W}_2$ denotes the
relaxed one-dimensional Willmore energy.

In the graph setting, minimizing sequences for Willmore-type energies may develop vertical parts in the limit. This naturally leads to $BV$-limits. For the two-dimensional Willmore graph problem and its $L^1$-relaxation we refer to Deckelnick, Grunau and Röger \cite{deckelnick2017minimising}. The one-dimensional weighted relaxation theorem of Dal Maso, Fonseca, Leoni and Morini \cite{maso2009higher} is another central ingredient for the present work, together with standard facts on $BV$-functions \cite{ambrosio2000functions}.

For $u\in BV((0,1))$, the distributional derivative decomposes as 
\[
    Du=(u')^a\,\mathcal L^1+D^j u+D^c u,
\]
where $D^j u$ is the jump part and $D^c u$ is the Cantor part. Jump parts have a clear geometric interpretation as vertical segments in graph limits. Cantor parts are more elusive: they are diffuse, meaning atomless, singular measures and one might expect them to be excluded by finite relaxed curvature energy. Recall that $SBV((0,1))$ denotes the subspace of $BV((0,1))$ consisting of functions whose Cantor part vanishes.

\textbf{The main purpose} of this paper is to show that this expectation is false. We
construct
\[     u\in BV((0,1))\cap C^0([0,1]) \quad \text{ such that } \quad  D^c u\neq0,   \qquad  \overline{\W}(u)<\infty.
\]
Hence finite relaxed one-dimensional Willmore energy does not imply $u\in SBV((0,1))$. The same mechanism also applies to the relaxed $p$-curvature energies $\overline{\W}_p$ for every $p>1$. Thus diffuse singular parts are not excluded by finite relaxed $L^p$-curvature energy either. We also want to give a geometrical interpretation of $u$, and study the regularity of the graph of $u$.

Relaxation of curvature-dependent energies has been studied from several viewpoints. For planar curvature functionals and elastica-type energies we refer to Bellettini, Dal Maso and Paolini, and to Bellettini and Mugnai \cite{bellettini1993semicontinuity,bellettini2004elastica,bellettini2007varifolds}. For Cartesian curves and $p$-elastic energies, the work of Acerbi and Mucci \cite{acerbi2017geometric,acerbi2017elastic} provides a geometric relaxation theory in which, for suitable continuous curves, the relaxed elastic energy is described by length plus $p$-curvature. In the case $p=1$, remarkably, the relaxed total curvature of a continuous Cartesian curve  does \emph{not} depend on the Cantor part of the derivative \cite{acerbi2017geometric}: Cantor parts are energetically neutral.  Weak notions of elastic energy for irregular rectifiable curves have been developed further by Mucci and Saracco \cite{mucci2023weak}, who showed that the relaxed $p$-elastic energy of a rectifiable curve is finite if and only if its arc-length parametrization belongs to $W^{2,p}$, see also the recent extension to curves in the sphere \cite{mucci2026sphere}. Our example fits precisely into this picture: the graph constructed below is an intrinsically tame $C^1\cap W^{2,2}$-curve in the sense of this theory, while the Cantor part appears only in one particular graph projection.

In the one-dimensional setting, the Willmore functional is the elastic energy of a curve. Boundary value problems and stability questions for the one-dimensional Willmore equation in graph form were studied by Deckelnick and Grunau \cite{deckelnick2007boundary,deckelnick2009stability}. Related boundary value, obstacle, threshold and $p$-elastic problems for curves and elastic graphs are treated, for instance, by Mandel \cite{mandel2015willmorecurves}, as well as in \cite{mueller2019obstaclecurves,mueller2023liyau, grunau2023obstacle,dallacqua2024pelasticobstacle,grunau2025optimality}. Closely related Willmore problems for surfaces of revolution, which reduce to one-dimensional profile curve problems, go back to Langer and Singer \cite{langer1984hyperbolic} and are studied, for instance, in \cite{dallacqua2008classical,dallacqua2011symmetric,eichmann2019grunau, schlierf2024dirichlet,dallacqua2025dimensionreduction}. The corresponding geometric flow is the elastic flow, for existence, computation, and long-time behaviour we refer to Dziuk, Kuwert and Sch\"atzle \cite{dziuk2002elastic}, Lin \cite{lin2012elasticflow}, Dall'Acqua, Lin, Pozzi and Spener \cite{dallacqua2017gradient,dallacqua2016lojasiewicz}, and to the survey of Mantegazza, Pluda and Pozzetta \cite{mantegazza2021surveyelastic}.

The expectation that finite energy excludes diffuse singular parts is also suggested by the theory of free discontinuity problems. In the framework of De Giorgi and Ambrosio \cite{degiorgi1988funzionale}, functionals such as the Mumford-Shah energy \cite{mumford1989optimal} are posed in $SBV$, and by Ambrosio's compactness theorem \cite[Thm.~4.8]{ambrosio2000functions} energy-bounded sequences cannot develop a Cantor part in the limit. The same philosophy persists for second-order functionals such as the Blake-Zisserman energy \cite{blake1987visual,carriero1997strong}. In contrast, the relaxed Willmore energy of graphs is a geometrically natural second-order functional whose finiteness does \emph{not} exclude a nontrivial Cantor part.

The mechanism in this work is that the Cantor part is concentrated exactly on the set where the absolutely continuous slope blows up. Thus the singular diffuse part is compatible with the relaxation theory. Geometrically, this also shows that the Cantor part of a graph derivative is projection-dependent: after a suitable rotation, the same curve may be represented by a Lipschitz graph without any singular derivative. This should be compared with the rectifiability theory for curvature varifolds. In particular, by second-order rectifiability, proved by Menne in \cite{menne2013second}, a graph with finite Willmore energy represented by a varifold is, up to an area null set, covered by a countable collection of $C^2$-manifolds. All this makes the following main theorem very surprising, since geometrically one would rather expect infinite relaxed Willmore energy because of the non-vanishing Cantor part. It seems that by naively applying projection techniques, one can get these irregular sets.

\subsection{Main Results}
We now state the condensed versions of the main results of the paper. We begin with the Willmore case for two reasons. First, this is the case from which the original motivation for the problem arose. Second, it is the simplest and geometrically most transparent instance of the construction, and therefore provides a convenient model before the same idea is formulated in the general weighted setting and, in particular, for $L^p$-curvature energies.  It  relies heavily on the relaxation theory of \cite[Prop.~2.3 and Thm.~3.4]{maso2009higher}. Part of these results can already be found in the second author's PhD thesis \cite{gulyak2024boundary}. 
\begin{thm}[Finite relaxed Willmore energy with Cantor part] \label{thm:main-finit-will}
There exists a function $u\in BV((0,1))\cap C^0([0,1])$ with $\overline \W(u)<\infty$ so that  $|D^c u |((0,1))> 0 $ and in particular $u\notin SBV((0,1))$.

The graph associated with $u$ admits an arc-length parametrization of class $C^1\cap W^{2,2}$, and its squared curvature energy equals $\W^a(u)$. Nevertheless, the classical curvature on the smooth graph pieces is unbounded near every point of the Cantor trace. Hence the graph is not locally a regular $C^2$-curve at any Cantor point.
\end{thm}
\begin{proof} For the full result, see Thm.~\ref{thm:finit_will} for the construction and the relaxed-energy estimate, and Thm.~\ref{thm:c1curv} for the geometric interpretation.
\end{proof}

\begin{figure}[h]
\centering
\includegraphics{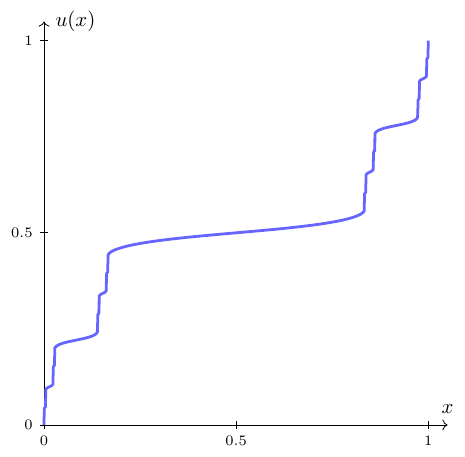}
\caption{{A rescaled version of the example, $u(x)=\frac12 U(x)/U(1)+\frac12 f_\delta(x)$, where $U$ and $f_\delta$ are constructed in the proof of Thm.~\ref{thm:finit_will}. The parameters are $s=\frac65$, $\delta=\frac16$, and $\beta=\frac25$. By Corollary~\ref{cor:no-positive-gap-pure-Willmore} this rescaled example still has finite relaxed Willmore energy.}}
\label{fig:plot}
\end{figure}

We also record several consequences and geometric interpretations of the construction. First, as explained in Corollary~\ref{cor:no-positive-gap-pure-Willmore}, there is no positive energy threshold for the pure relaxed Willmore term which excludes Cantor parts. The example of Thm.~\ref{thm:main-finit-will} can be rescaled so that the relaxed Willmore energy becomes arbitrarily small while the Cantor part remains nontrivial.

Second, the Cantor part is not an intrinsic singularity of the underlying curve. It is a feature of the chosen projection direction. As discussed in Remark~\ref{rem:rot}, after a suitable clockwise rotation, the same geometric curve can be represented as a Lipschitz graph. The derivative of this rotated graph function has no singular part.

There is also a complementary $C^2$-variant of the construction, see Remark~\ref{rem:C2-variant}. By modifying the singular profile, one can arrange that the curvature of the smooth graph pieces extends continuously across the Cantor set by setting it equal to zero there. In that case the graph is a one-dimensional $C^2$-submanifold of $\mathbb R^2$, although the original graph function still satisfies $D^c u\neq0$. Thus Cantor parts in graph projections may occur even for geometrically $C^2$-regular curves.

This should be compared with Menne's second-order rectifiability theorem
\cite[Thm.~3.6]{menne2013second}, see Remark~\ref{rem:menne}. There is no contradiction: the theorem gives a measure-theoretic $C^2$-covering of the associated varifold, whereas the Cantor part appears in a particular graph projection.


The analytic mechanism behind the construction is more general. We prove the following weighted version, extending the Cantor-part examples from \cite[Prop.~2.3 and Remark~3.2 (iv)]{maso2009higher} to all exponents $p>1$.

\begin{thm}[General weighted version]
\label{thm:main-Xpsi}
Let $p>1$, and let $\psi:\mathbb R\to(0,+\infty)$ be a bounded Borel weight satisfying the standard assumptions of the relaxation theory. If, for some
$\alpha>2p-1$,
\[
    \psi(t)\le C t^{-\alpha}
    \qquad\text{for all }t\ge1,
\]
then there exists
\[
    u\in X_\psi^p((0,1))
\quad \text{
such that } \quad
    D^c u\neq0 .
\]
\end{thm}
\begin{proof}
    For the notation see Section~\ref{sec:preliminaries}. For the proof, see Thm.~\ref{thm:cantor_part_Xp}. 
\end{proof}
Thm.~\ref{thm:main-Xpsi} extends \cite[Prop.~2.3 and Remark~3.2~(iv)]{maso2009higher}, where such examples were obtained only for exponents $p\in(1,2)$ sufficiently close to $1$, to all exponents $p>1$.

The  construction applies to relaxed $L^p$-curvature energies of $\W_p$ for $ p>1$. Indeed, the corresponding weight has decay exponent $3p-1>2p-1$, so the Cantor-profile construction is compatible with the relaxation theorem for $\mathcal F_p$. Consequently, finite relaxed $L^p$-curvature energy does not exclude a nontrivial Cantor part of the distributional derivative in some projection. That means that one cannot avoid these phenomena just by choosing the power of the curvature contribution high enough. 

\begin{corollary}[Cantor parts for \texorpdfstring{$L^p$}{Lp}-curvature energies]
Let $p>1$.  Then there exists a function $ u\in BV((a,b))\cap C^0([a,b])$ such that
\[
    D^c u\neq 0
    \qquad\text{and}\qquad
    \overline{\W}_p(u)<\infty .
\]
In particular, finite relaxed $L^p$-curvature energy does not imply
$u\in SBV((a,b))$. 
\end{corollary}
\begin{proof}
    See the proof of Corollary~\ref{cor:Lp-curvature-cantor} 
\end{proof}

Geometrically, as in the Willmore case, the graph associated with the constructed function admits an arc-length parametrization and its weak scalar curvature belongs to $L^p(0,L)$. Hence the Cantor part in the derivative is compatible with finite $L^p$-curvature of the underlying curve. By Remark~\ref{rem:C1W2pcurves}, the graph can be chosen as a $C^2$-curve, although the corresponding graph function still has $D^c u\neq0$.

Finally, the $C^2$-variant has a surface-of-revolution consequence. After adding a sufficiently large constant to make the profile strictly positive, one may revolve the resulting $C^2$-curve around the $x$-axis. This gives a $C^2$-surface of revolution with finite Willmore energy, while the generating radius function belongs to $BV\cap C^0$ and still has a nontrivial Cantor part in its distributional derivative.

\subsection{Further research}

To the authors’ knowledge, it is still open whether \emph{Willmore minimizers} may have a nonzero Cantor part. More precisely, one may ask whether there are boundary conditions, obstacle constraints, or lower-order terms for which a minimizer of the relaxed one-dimensional Willmore functional has a nonvanishing Cantor part. Conversely, it would be important to characterize additional assumptions under which finite relaxed energy, together with minimality, forces $D^c u=0$, or even implies higher regularity away from jump points.

This question is particularly relevant in connection with the \emph{two-dimensional graph problem} studied by Deckelnick, Grunau and Röger \cite{deckelnick2017minimising}. In that setting, minimizing sequences for the Willmore functional may develop vertical parts, and the relaxed functional naturally acts on $BV$-limits. It remains unclear whether genuinely diffuse singular parts can occur in minimizers, or whether some variational structure rules them out.  Although our construction is one-dimensional, it can be embedded into the two-dimensional graph setting by a product construction. Let $u\colon[0,1]\to \R$ have finite one-dimensional  Willmore energy as a graph, then we can extend it to a function $\overline{u}\colon[0,1]^2\to \R$ with finite two-dimensional Willmore energy as a graph. We simply set for all $(x,y)\in[0,1]^2 \colon \overline{u}(x,y)=u(x)$. Then \(\operatorname{graph}[\bar u]\) is the product of the planar graph of \(u\)
with \([0,1]\). Its principal curvatures are \(\kappa_u\) and \(0\). Hence,
with the convention \(H=(k_1+k_2)/2\),
\begin{align*}
    \mathcal W(\bar u)
    = \int_{\operatorname{graph}[\bar u]} H^2\,\diff A
    = \frac14 \int_{\operatorname{graph}[u]}\kappa_u^2\,\diff s
    = \frac14\,\mathcal W(u).
\end{align*}
By the slicing theorem for $BV$-functions \cite[Theorems~3.107 and~3.108]{ambrosio2000functions}, or directly by testing against smooth compactly supported vector fields, one obtains $  D^c\bar u = \,D^c u\otimes\mathcal L^1\mrest(0,1)$. In particular, if $D^c u\neq0$, then $D^c\bar u\neq0$.

A related question is whether the present construction can be modified so that the regular pieces of the graph are not only finite-energy pieces, but \emph{critical points of the smooth Willmore functional}. In the present proof, the singular profile in the absolutely continuous part is chosen in order to guarantee $w\in L^1((0,1))$ and $  \Psi_2\circ w\in W^{1,2}((0,1))$, and to make the Cantor part compatible with the blow-up set of $w$.  It would be interesting to know whether one can choose profiles $w$ on the complementary intervals of the Cantor set such that the corresponding graph pieces solve the stationary Willmore equation and still satisfy $w(x)\to+\infty$ at the endpoints of the intervals. Such a construction would produce a one-dimensional relaxed Willmore critical point with nontrivial Cantor part. 

Another natural direction is to characterize the singular continuous measures that may arise as Cantor parts of distributional derivatives of functions with finite pure relaxed Willmore energy. Finally, to the best of our knowledge, the exact relaxation of the pure Willmore term remains open. The relaxation formula of Dal Maso, Fonseca, Leoni and Morini~\cite{maso2009higher} applies to the augmented functional containing the total variation term, but it does not directly yield a formula for the lower semicontinuous envelope of the pure Willmore energy alone.

\section{Preliminaries}\label{sec:preliminaries}   
In this section we collect the background for the proof of Thm.~\ref{thm:finit_will} below. We first recall the theory of $BV$ functions of one variable, with emphasis on the decomposition of the distributional derivative into absolutely continuous, jump and Cantor part. We then present the relaxation theorem of Dal Maso, Fonseca, Leoni and Morini \cite{maso2009higher}, which for $p=2$ identifies the relaxation of the one-dimensional Willmore energy augmented by a total variation term and is the central tool of our construction.

\subsection*{Functions of bounded variation}        
We briefly recall the theory of $BV$ functions of one variable, following \cite[Subsection 2.1]{maso2009higher} and \cite{ambrosio2000functions}. A function $u \in L^{1}((a,b))$ belongs to $BV((a,b))$ if and only if its \emph{total variation}
\begin{align*}
\operatorname{V}(u,(a,b))
:=\sup \left\{\left. \int_{a}^{b} u\, \varphi' \diff x\ \right|\
\varphi \in C_{0}^{1}((a,b)) \text{ and } \|\varphi\|_\infty \leq 1\right\}
\end{align*}
is finite. In this case the distributional derivative $Du$ of $u$ is a bounded Radon measure on $(a,b)$, and $\left|Du\right|$ denotes its total variation measure. In particular, by \cite[Proposition 3.6]{ambrosio2000functions} it holds $\left|Du\right|((a,b))=\operatorname{V}(u,(a,b))$. By the Lebesgue decomposition with respect to $\Leb^1$, the measure $Du$ splits into an absolutely continuous part and a singular part, 
\begin{align*}
Du=(u')^{a}\, \Leb^{1} + D^{s}u
=(u')^{a}\, \Leb^{1} + D^{j}u + D^{c}u,
\end{align*}
where $D^{j}u$ denotes the jump part and $D^{c}u$ the Cantor part of $Du$, which we characterize below. Every function $u \in BV((a,b))$ is $\Leb^{1}$-a.e.\ differentiable in $(a,b)$, and for $\Leb^{1}$-a.e.\ $x \in (a,b)$ the classical derivative coincides with $(u')^{a}(x)$. Moreover, for every $u \in BV((a,b))$ the left and right approximate limits
\begin{align*}
u_-(y):=\lim_{\varepsilon \searrow 0} \frac{1}{\varepsilon}
\int_{y-\varepsilon}^{y} u(x) \diff x
\quad \text{ and } \quad
u_+(y):=\lim_{\varepsilon \searrow 0} \frac{1}{\varepsilon}
\int_{y}^{y+\varepsilon} u(x) \diff x
\end{align*}
exist at every point $y \in (a,b)$ and define a left-continuous and a right-continuous function, respectively. The functions $u_-$ and $u_+$ coincide $\Leb^1$-a.e.\ in $(a,b)$. The exceptional set where they differ,
\begin{align*}
S_{u}:=\left\{y \in (a,b) \, \big| \, u_-(y) \neq u_+(y)\right\},
\end{align*}
is called the set of essential discontinuities, or jump set, of $u$, it is at most countable. With the counting measure $\mathcal{H}^{0}$ concentrated on $S_u$, the singular part can be written as
\begin{align*}
D^{s}u=(u_+-u_-)\, \mathcal{H}^{0}\mrest S_{u}+D^{c}u,
\end{align*}
that is, the jump part $D^ju=(u_+-u_-)\, \mathcal{H}^{0}\mrest S_{u}$ is a purely atomic measure, while the Cantor part $D^c u$ is a singular diffuse measure.
           
There is also a pointwise notion of variation for functions defined everywhere. We recall that $u:(a,b)\rightarrow \mathbb{R}$ has bounded pointwise variation $\operatorname{pV}(u,(c,d))$ on an interval $(c,d)\subset(a,b)$ if
\begin{align*}
\operatorname{pV}(u,(c,d)):=\sup \sum_{i=1}^{k}\left|u\left(y_{i}\right)-u\left(y_{i-1}\right)\right|<+\infty,
\end{align*}
where the supremum is taken over all finite families of points $(y_{0}, y_{1}, \ldots, y_{k})$ with $c<y_{0}<y_{1}<\cdots<y_{k}<d$ and $k \in \mathbb{N}$. The approximate limits $u_-$ and $u_+$ defined above are in fact precise and good representatives of $u\in BV((a,b))$, in the sense that
for every interval $(c,d)\subset(a,b)$
\begin{align*}
\left|Du\right|((c,d))=\operatorname{pV}(u_-,(c,d))=\operatorname{pV}(u_+,(c,d)).
\end{align*}
Conversely, if a function $u\in L^1((a,b))$ has bounded pointwise variation in $(a,b)$, then it belongs to $BV((a,b))$, with $\left|Du\right|((c,d)) \leq \operatorname{pV}(u,(c,d))$ for every interval $(c,d)\subset(a,b)$.

Finally, we recall a compactness result for $BV$ spaces; see
Ambrosio-Fusco-Pallara~\cite[Theorem~3.23]{ambrosio2000functions}. If $\{u_k\}_{k=1}^\infty$ is a sequence in $BV((a,b))$ satisfying the condition 
\begin{align*}
\sup_{k\in \N} \left\{ \|u_k\|_{L^1((a,b))}+ |Du_k|((a,b))\right\}<\infty,
\end{align*}
then there exist a subsequence $\{u_{k_j}\}_{j=1}^\infty$ and a limit function $u \in BV((a,b))$ such that
\begin{align*}
u_{k_j}\to u \ \text{ in } L^1((a,b)) \text{ as } j\to\infty
\quad \text{ and } \quad \lim_{j\to\infty}\int_a^b \varphi \diff Du_{k_j} = \int_a^b \varphi \diff Du
\ \text{ for all } \varphi\in C_0^0((a,b)),
\end{align*}
i.e.\ the derivatives $Du_{k_j}$ converge weakly-$*$ to $Du$ in the sense of measures.

\subsection*{The relaxation theorem of Dal Maso-Fonseca-Leoni-Morini}
\label{subsec:maso}
 
Let us now focus on some results from \cite{maso2009higher}. There the authors consider the following functional, arising from a total variation-based model for image restoration involving a second-order term that eliminates the staircase effect. For an exponent $p \in (1,+\infty)$ let $\mathcal{F}_{p}: L^{1}((a,b)) \rightarrow[0,+\infty]$ be defined by
\begin{align*}
\mathcal{F}_{p}(u):= \begin{cases}\displaystyle\int_{a}^{b}\left|u'\right| \diff x
+\int_{a}^{b} \psi\left(u'\right)\left|u''\right|^{p} \diff x
& \text{ if } u \in W^{2,p}((a,b)), \\[1ex]
+\infty & \text{ otherwise,}\end{cases}
\end{align*}
where $\psi: \mathbb{R} \rightarrow (0,+\infty)$ is a bounded Borel function to be specified below. In particular, we extend $\mathcal{F}_{p}$ to $L^{1}((a,b))$ by setting $\mathcal{F}_{p}(u):=+\infty$ if $u \in L^{1}((a,b)) \setminus W^{2,p}((a,b))$. Using the theory of relaxation, one identifies the lower semicontinuous envelope of $\mathcal F_p$ with respect to strong $L^{1}$-convergence: for every $u\in L^1((a,b))$ we set
\begin{align}
\overline{\mathcal{F}}_p(u):= \inf \left\{ \left. \liminf_{k\to\infty}
\mathcal{F}_p(u_k) \, \right| \, u_k\to u \text{ in } L^1((a,b)) \right\},
\label{eq:overF_p}
\end{align}
where the infimum is taken over all sequences $\{u_k\}_{k\in\N}$ in $L^1((a,b))$ with $u_k\to u$ in $L^1((a,b))$.

At this point, let us briefly compare $\mathcal{F}_p$ with the Willmore energy. For $p=2$ we obtain
\begin{align}
\mathcal{F}_2(u)=\W(u) + \int_{a}^{b}\left|u'\right| \diff x,
\quad \text{ with } \quad
\psi(\tau):= \frac{1}{(1+\tau^2)^{5/2}} \quad \text{ for all } \tau \in \R.
\label{eq:f2will}
\end{align}
Thus, for $p=2$, the second-order term in $\mathcal F_2$ is exactly the one-dimensional Willmore energy. The relaxation theorem for $\mathcal F_2$ therefore provides a useful framework for constructing finite-energy approximating sequences, although the pure functional $\W$ does not contain the total variation term $\int_0^1 |u'| \diff x$. 

In this context, it is important to notice that, even under the boundary conditions, finiteness of $\W(u)$ alone does not yield an $L^1((0,1))$-bound on $u'$. Such a bound follows only under an additional smallness assumption on $\W(u)$. Without this assumption, the graph in the relaxed setting may contain vertical segments of arbitrarily large length that are not penalized by $\W(u)$.  

The underlying obstruction is classical in the theory of free elastica and is often treated as folklore; for a convenient modern reference, see~\cite[Appendices A and B]{deckelnick2025basin}. The first nontrivial stationary free elastica are given by half-periods of Euler's rectangular elastica and occur at the scale-invariant energy level $2\pi$. In their explicit arc-length parametrization, the horizontal velocity vanishes at the midpoint of a half-period. Thus the graphical parametrization develops a vertical tangent at this first nontrivial threshold.

But first, we impose two additional conditions on the bounded Borel function $\psi: \mathbb{R} \rightarrow (0,+\infty)$, namely
\begin{align}
M:=\int_{-\infty}^{+\infty}(\psi(t))^{1/p} \diff t<+\infty
\quad \text{ and } \quad
\inf_{t \in K} \psi(t)>0 \ \text{ for every compact set } K \subset \mathbb{R}.
\label{eq:psibed}
\end{align}
If we now define $\Psi_{p}: \overline{\mathbb{R}} \rightarrow[0,M]$ as the primitive of $\psi^{1/p}$ by
\begin{align}
\Psi_{p}(t):=\int_{-\infty}^{t}(\psi(s))^{1/p} \diff s,
\label{eq:psi_p}
\end{align}
and while $\Psi_{p}^{-1}:[0,M] \rightarrow \overline{\mathbb{R}}$  denotes its inverse function, then for every $u\in W^{2,p}((a,b))$ we obtain
\begin{align}
\mathcal{F}_p(u)= \int_a^b |u'| \diff x
+\int_a^b \left|\frac{\diff\ }{\diff x}\Big( \Psi_p \circ u' \Big)(x)\right|^p \diff x.
\label{eq:F_prew}
\end{align}

In \cite[Theorem 3.4]{maso2009higher} the authors characterize the relaxation of the functional $\mathcal{F}_p$ with respect to strong convergence in $L^1((a,b))$. To this end, they introduce the subspace of $L^{1}$-functions that can be approximated by $\mathcal{F}_p$-bounded sequences, which they call $X_{\psi}^{p}((a,b))$ \cite[Definition 3.1 and Remark 3.2]{maso2009higher} and which we recall now. In view of \eqref{eq:F_prew}, one of the properties of a function $u \in X_{\psi}^{p}((a,b))$ has to be $v:=\Psi_{p} \circ (u')^{a}\in W^{1,p}((a,b))$. Since $\Psi_p^{-1}$ is continuous and $v\in C^0([a,b])$ by Sobolev embedding, it follows that $\Psi_{p}^{-1}(v)$ is continuous on $[a,b]$ with values in $\overline{\mathbb{R}}$.  Therefore, there is a continuous \(\overline{\mathbb R}\)-valued representative of the  absolutely continuous density $(u')^{a}$ induced by $\Psi_{p}^{-1}(v)$. Hence, in all pointwise statements  for $(u')^{a}$ below, we use the chosen continuous representative of the absolutely continuous so-called Radon-Nikodym density $(u')^{a}$.

Next,the sets where the continuous \(\overline{\mathbb R}\)-valued representative of the absolutely continuous part $(u')^a$ blows up are denoted  by
\begin{align*}
Z^{+}\left[(u')^{a}\right]:=\left\{x \in (a,b) \, \left|\, (u')^{a}(x)= +\infty \right.\right\},
\quad
Z^{-}\left[(u')^{a}\right]:=\left\{x \in (a,b) \, \left|\, (u')^{a}(x)= -\infty \right.\right\}.
\end{align*}
Then we define
\begin{align*}
X_{\psi}^{p}((a,b)):= \left\{ u\in BV((a,b)) \, \left|\,
\Psi_{p} \circ (u')^{a}\in W^{1,p}((a,b)) , \
\left(D^su\right)^{\pm} \text{ is concentrated on } Z^{\pm}\left[(u')^{a}\right]
\right.\right\}.
\end{align*}

With the set $X_{\psi}^{p}((a,b))$ at hand, the authors of \cite[Theorem 3.4]{maso2009higher} were able to characterize the relaxation of $\mathcal{F}_{p}$ with respect to strong convergence in $L^{1}((a,b))$ defined in \eqref{eq:overF_p}:
 
\begin{thm}[Dal Maso-Fonseca-Leoni-Morini]
\label{thm:maso}
It holds
\begin{align*}
\overline{\mathcal{F}}_{p}(u)= \begin{cases}\displaystyle
\left|Du\right|((a,b))+\int_{a}^{b}\left|v'\right|^{p} \diff x
& \text{ if } u \in X_{\psi}^{p}((a,b)), \\[1ex]
+\infty & \text{ otherwise,}\end{cases}
\end{align*}
where $v:=\Psi_{p} \circ (u')^{a}$, so that the higher-order term depends only on $(u')^{a}$.
\end{thm}

As in \cite[Remark 3.2]{maso2009higher}, let us discuss some properties of functions $u\in X_{\psi}^{p}((a,b))$. For every jump point $x_{0}$ with $u_+(x_{0})-u_-(x_{0})>0$ it holds $\lim_{x \rightarrow x_{0}}(u')^{a}(x)=+\infty$, and for every jump point $x_{0}$ with $u_+(x_{0})-u_-(x_{0})<0$ we have $\lim_{x \rightarrow x_{0}}(u')^{a}(x)=-\infty$. Furthermore, if $Du$ has a non-vanishing Cantor or jump part, then $(u')^{a}$ cannot be bounded. In particular, piecewise constant functions with jumps and the Cantor function are excluded from $X_{\psi}^{p}((a,b))$, since for them $Z^{\pm}\left[(u')^{a}\right]=\varnothing$. The same applies to polygons. For the case $p=2$, this means that the Cantor function (for its definition, see the proof of Thm.~\ref{thm:finit_will} and Fig.~\ref{fig:cantor-delta-02}) and piecewise smooth graphs with corners have infinite relaxed $\mathcal F_2$ energy.
 

In view of the Cantor function example, one might expect that all functions with non-trivial Cantor part are excluded from $X_{\psi}^{p}((a,b))$. Surprisingly, in \cite[Remark 3.2 (iv)]{maso2009higher} Dal Maso, Fonseca, Leoni and Morini constructed functions with non-trivial Cantor part in $X_{\psi}^{p}((a,b))$ for exponents $p\in(1,2)$ sufficiently close to $1$, provided that $\psi$ satisfies 
\begin{align*}
\psi(t) \leq c\, t^{-\alpha} \quad \text{ for all } t \geq 1
\end{align*}
and some $c>0$, $\alpha>1$. The function $\psi$ corresponding to the Willmore term in $\mathcal{F}_2$, defined in \eqref{eq:f2will}, satisfies this condition with $\alpha=5$, as well as condition \eqref{eq:psibed}. In this paper, we extend this result to $p=2$, hence to the relaxed Willmore energy $\overline{\W}$, by constructing a function $u$ with $\overline{\W}(u)<+\infty$ and non-vanishing Cantor part. This is indeed a surprising result, since $p=2$ is not close to $1$ and is therefore not covered by \cite[Remark 3.2 (iv)]{maso2009higher}.

\section{Finiteness of the Relaxed Willmore Energy Does Not  Imply \texorpdfstring{$SBV$}{SBV}}
By the relaxation theorem of Dal Maso-Fonseca-Leoni-Morini (see Subsection~\ref{subsec:maso}), a singular part of $Du$ is admissible only if it is concentrated on the blow-up set of the absolutely continuous slope $(u')^a$. The pure Cantor function fails this criterion, since its absolutely continuous derivative vanishes. We therefore superpose a generalized Cantor function $f_\delta$ with a continuous monotone function $U$ whose derivative tends to $+\infty$ precisely on the Cantor set, while $\Psi_2\circ(u')^a$ remains in $W^{1,2}((0,1))$. For $u:=U+f_\delta$ this yields $\overline\W(u)\le \W^a(u)<\infty$ despite $D^c u\neq0$. Two remarks complement the theorem: the example can be rescaled so that its relaxed Willmore energy becomes arbitrarily small (Corollary~\ref{cor:no-positive-gap-pure-Willmore}), and the Cantor part is a projection effect that disappears after rotating the graph (Remark~\ref{rem:rot}).

\begin{thm}\label{thm:finit_will}
There exists a function $u\in BV((0,1))$ such that $\overline{\W}(u)<\infty$ and $|D^c u|((0,1))>0$. In particular, $u\notin SBV((0,1))$ 
with 
\[\overline{\W}(u)\le \W^a(u)=
\int_0^1
\left|
\frac{\diff\ }{\diff x}\Psi_2((u')^a(x))
\right|^2 \diff x.\]
\end{thm}

            \begin{proof}
            By the characterization in \cite{maso2009higher}, the Cantor function $f_\delta:[0,1]\to[0,1]$, introduced in Step \cbm 2, cannot serve as our example. Indeed, its positive singular derivative is concentrated on a Cantor set $\mathbb D_\delta$, defined in \eqref{eq:defcantorset}, whereas $Z^+[(f_\delta')^a]=\varnothing$ and therefore  $\overline {\mathcal F}_2(f_\delta)=+\infty$.

To overcome this difficulty, we add a continuous function $U:[0,1]\to\mathbb R$, constructed in Steps \cbm 3-\cbm 6, and define $u:=f_\delta+U.$ The function $U$ is chosen in such a way that $\mathbb D_\delta\subset Z^+[(u')^a]$,  therefore obeying $(U')^a=+\infty$ on the Cantor set, while at the same time its one-dimensional Willmore energy remains finite. Thus, $f_\delta$ generates the Cantor part of the derivative, whereas $U$ ensures that this singular part is compatible with finite relaxed Willmore energy.

To prove that $\overline{\W}(u)<\infty$, we construct in Step \cbm 7 a sequence $(u_k)_{k\in\mathbb N}\subset W^{2,2}((0,1))$ by using Thm.~\ref{thm:maso}.
            
\medskip

\noindent\cb{1} We first recall the auxiliary function
\begin{align*}
\Psi_2(t):=\int_{-\infty}^t \frac{1}{(1+\tau^2)^{5/4}} \diff \tau,
\qquad
M:=\Psi_2(\infty)=\int_{\mathbb R}\frac{1}{(1+\tau^2)^{5/4}} \diff \tau
=\frac{\sqrt{\pi}\,\Gamma\!\left(\frac34\right)}{\Gamma\!\left(\frac54\right)}.
\end{align*}
where $\Gamma$ denotes the Gamma function. Following \cite[Lemma 1]{deckelnick2007boundary}, $\Psi_2$ extends to a continuous,
strictly increasing map
\[
\Psi_2\colon \overline{\mathbb R}\to[0,M],
\qquad
\Psi_2^{-1}:[0,M]\to\overline{\mathbb R}.
\]
 We also set the function $\psi\colon \R \to (0,\infty)$ by
\[
\psi(\tau):=\frac{1}{(1+\tau^2)^{5/2}}, \qquad \tau\in\mathbb R.
\]
For $u\in W^{2,2}((0,1))$,  we recall
the one-dimensional Willmore functional 
\begin{align*}
\W(u)
&:= \int_0^1 \psi(u')\,|u''|^2 \diff x
 = \int_0^1 \frac{|u''(x)|^2}{(1+u'(x)^2)^{5/2}} \diff x \\
&= \int_0^1 \kappa_u(x)^2\sqrt{1+u'(x)^2} \diff x
 = \int_{\operatorname{graph}[u]} \kappa_u^2 \diff s \\
&= \int_0^1 \left|\frac{\diff}{\diff x}\Psi_2(u'(x))\right|^2 \diff x .
\end{align*}
Here $\kappa_u=u''(1+u'^2)^{-3/2}$ denotes the curvature of the graph of $u$.

\medskip
\noindent\cb{2} Following \cite[pp.~2356 ff.]{maso2009higher}, we now construct the generalized Cantor set $\mathbb D_\delta$ and the associated Cantor function $f_\delta$.
Let 
\begin{align}
    \delta\in\left(0,\frac12\right)
    \label{eq:1dcond1}\tag{C1}
\end{align}
with $\delta$  be suitably chosen below. Its precise value will be chosen later.
Starting from the interval $[0,1]$, we first remove the open middle interval of length $1-2\delta$, namely
\[
I_{11}:=(\delta,1-\delta),
\quad \text{ therefore }
\quad
[0,1]\setminus I_{11}=[0,\delta]\cup[1-\delta,1].
\]

In the second step, we remove from each of these two intervals its open middle subinterval of length $\delta(1-2\delta)$, namely
\[
 	I_{21}  := \big(\delta^2,\ \delta(1-\delta)\big), \quad 
            	I_{22} := \big((1-\delta)+\delta^2, \ (1-\delta) + \delta (1-\delta)\big)
            	= (1-\delta+\delta^2,\ 1-\delta^2).
\]
 We repeat the  procedure on $[0,1]\setminus( I_{11}\cup I_{21}\cup I_{22})$ and remove four open intervals $I_{31}, I_{32}, I_{33}, I_{34}$. Iterating this construction, after $j-1$ steps there remain $2^{j-1}$ closed intervals, each of length $\delta^{j-1}$. In step $j$, we remove from the middle of each of these intervals an open subinterval of length $\delta^{j-1}(1-2\delta)$. We call these intervals
\[
I_{jk},\qquad k=1,\dots,2^{j-1},
\]
ordered from left to right. In particular,
$
I_{j1}=\big(\delta^j,\delta^{j-1}(1-\delta)\big).
$

We thus define the generalized Cantor set by
\begin{align}
    \mathbb D_\delta
    :=\bigcap_{\ell=1}^\infty C_\ell
  \quad  \text{with}
\quad    C_\ell:= [0,1]\setminus \bigcup_{j=1}^\ell \bigcup_{k=1}^{2^{j-1}} I_{jk}, \ell\in \N.
  \label{eq:defcantorset}
\end{align}
Then $C_\ell$ consists of $2^\ell$ intervals, each of length $\delta^\ell$. Hence
\[
\mathcal L^1(C_\ell)=2^\ell\delta^\ell=(2\delta)^\ell \xrightarrow[\ell\to\infty]{} 0,
\]
and therefore $\mathbb D_\delta=\bigcap_{\ell=1}^\infty C_\ell$ has Lebesgue measure zero.

We now define the approximating densities and functions by
\[
g_\ell:=\frac{1}{(2\delta)^\ell}\,\mathbbm 1_{C_\ell}
=
\frac{1}{(2\delta)^\ell}
\left(
1-\sum_{j=1}^\ell\sum_{k=1}^{2^{j-1}}\mathbbm 1_{I_{jk}}
\right),
\qquad
f_\ell(x):=\int_0^x g_\ell(\xi) \diff \xi.
\]

By construction, $f_\ell$ is affine on each connected component of $C_\ell$ and constant on every interval $I_{jk}$ with $1\le j\le \ell$. More precisely, if $I_{jk}$ denotes the $k$-th removed interval at level $j$, counted from left to right, then
\begin{align}
    f_\ell\equiv \frac{2k-1}{2^j}
\qquad\text{on } I_{jk}. \label{eq:f_ell_formula}
\end{align}

On each of the $2^\ell$ components of $C_\ell$, the function $f_\ell$ is linear with slope $(2\delta)^{-\ell}$.
Since $g_\ell\ge 0$, each $f_\ell$ is nondecreasing. Moreover,
\[
f_\ell(1)
=
\int_0^1 g_\ell \diff x
=
\frac{1}{(2\delta)^\ell}\,\mathcal L^1(C_\ell)
=
\frac{1}{(2\delta)^\ell}(2\delta)^\ell
=
1.
\]
Together with $f_\ell(0)=0$, this shows that $f_\ell:[0,1]\to[0,1]$.

By Lemma~\ref{lem:Cantor}, the sequence $(f_\ell)$ converges uniformly to a continuous function $f_\delta:[0,1]\to[0,1]$  as $\ell\to\infty$, namely the generalized Cantor function associated with $\mathbb D_\delta$. Moreover,
\[
(f_\delta')^a = 0 \ \text{ a.e. in }(0,1), \text{ and }\ Df_\delta=D^sf_\delta \text{ as Radon measures,} 
\]
whereas the Cantor part $D^cf_\delta$ is a nonnegative measure supported on $\mathbb D_\delta$. In particular, all variation of $f_\delta$ is concentrated on $\mathbb D_\delta$. The function $f_\delta$ is nondecreasing, satisfies
\[
f_\delta(0)=0,\qquad f_\delta(1)=1,
\]
and is constant on every connected component of $[0,1]\setminus \mathbb D_\delta$.


\begin{figure}[H]
                \centering
\includegraphics[]{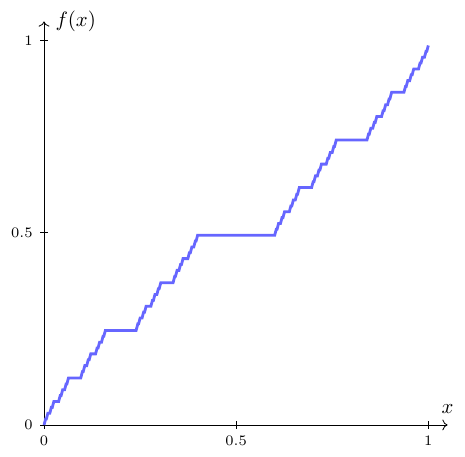}
                \caption{$f_\delta$ for $\delta=0.4$}
                \label{fig:cantor-delta-02}
            \end{figure}

\medskip 
\noindent\cb{3} We now construct an auxiliary derivative $w=w_\delta$ for the function $U$ that will later be added to $f_\delta$. The Cantor part of the derivative will still come entirely from $f_\delta$. The role of $U$ is to modify only the absolutely continuous part in such a way that the resulting function has finite relaxed Willmore energy. For the moment, we therefore define only the derivative candidate
\[
w_\delta:[0,1]\to[0,\infty],
\]
which we require to be continuous, to satisfy
\[
w_\delta(x)=+\infty \quad \Longleftrightarrow \quad x\in\mathbb D_\delta,
\]
and to obey
\[
\Psi_2\circ w_\delta\in W^{1,2}((0,1)).
\]
Once $U$ is defined as a primitive of $w_\delta$, this condition will yield finite curvature energy for $U$. The specific singular profile $\Phi$ used below and differs from the construction in \cite{maso2009higher}.
We fix
\begin{align}
    \beta\in\left(\frac13,\frac{1}{2}\right),
    \label{eq:1dcond2}\tag{C2}
\end{align}
where the upper bound $\beta<\frac12$ has two roles. First, it implies $\beta<1$ and therefore guarantees the integrability of the singularities of the profile at the endpoints $0$ and $1$ of the reference interval. Second, it is responsible for the unboundedness of the curvature near the endpoints of the removed intervals, used in Definition \eqref{eq:def_yj}. For comparison, see also Remark~\ref{rem:C2-variant}. The lower bound $\beta>\frac13$ will be needed later in the local Willmore estimates on the intervals $I_{jk}$ in Step~\cbm 5.

We define
\[
\Phi:[0,1]\to[0,\infty],\qquad
\Phi(x):=C_\beta\left(\frac{1}{x^\beta(1-x)^\beta}-4^\beta\right),
\]
where the normalizing constant $C_\beta>0$ is chosen so that
\[
\int_0^1 \Phi(x) \diff x=1.
\]
Since $x(1-x)\le \frac14$ on $[0,1]$, we have
$
{x^\beta(1-x)^\beta}\le 4^{-\beta},
$
with equality only at $x=\frac12$. Hence $\Phi\ge 0$ and $\Phi(\frac12)=0$, which will later imply monotonicity of its primitive.  Moreover, because $\beta<1$, the singularities at the endpoints are integrable, and
\begin{align*}
    C_\beta^{-1}
    &= \int_0^1 \left(\frac{1}{x^\beta(1-x)^\beta}-4^\beta\right) \diff x \\
    &= \Beta(1-\beta,1-\beta)-4^\beta
     = \frac{\Gamma(1-\beta)^2}{\Gamma(2-2\beta)}-4^\beta>0,
\end{align*}
where $\Beta$ denotes Euler's Beta function \cite[(5.12.1)]{NIST:DLMF}. Up to a horizontal rescaling and an additive constant, $\Phi$ will serve as the basic profile for the absolutely continuous part of the derivative. Its primitive is given for $x\in [0,1]$ by
\begin{align*}
    C_\beta^{-1}\int_0^x \Phi(s) \diff s
    = \Beta_x(1-\beta,1-\beta)-4^\beta x
    = \frac{x^{1-\beta}}{1-\beta}\,{}_2F_1(1-\beta,\beta;2-\beta;x)-4^\beta x,
\end{align*}
where $\Beta_x$ is the incomplete Beta function and ${}_2F_1$ denotes the Gauss hypergeometric function \cite[(8.17.7)]{NIST:DLMF}.

\begin{figure}[H]
\centering
\includegraphics{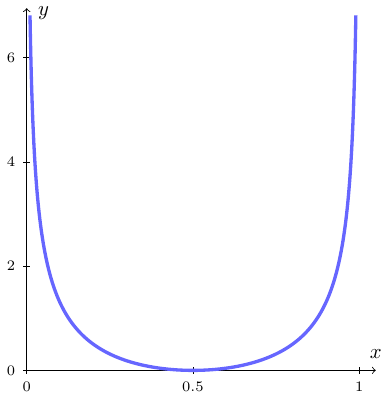}
\hspace{2em}
\includegraphics{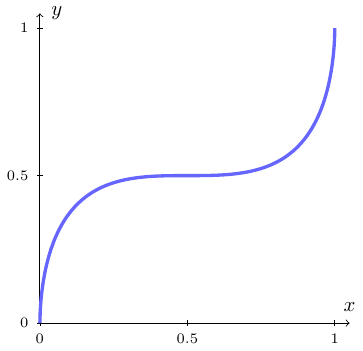}
\caption{Left: $\Phi(x)$ for $\beta=\frac25$, right: its primitive $x\mapsto \int_0^x \Phi(s)\diff s$ for $\beta=\frac25$.}
\end{figure}

Furthermore, $\Phi$ is strictly convex on $(0,1)$. Indeed,
            \begin{align*}
            	\frac{\Phi'(x)}{C_\beta} &=  \frac{-\beta }{x^{\beta+1}(1-x)^\beta}+\frac{\beta }{x^{\beta}(1-x)^{\beta+1}} 
            	=  \frac{\beta\big[-(1-x)+x\big] }{x^{\beta+1}(1-x)^{\beta+1}}
            		=  \frac{\beta\big[2x-1\big] }{x^{\beta+1}(1-x)^{\beta+1}}, \\
            		 \frac{\Phi''(x)}{C_\beta} &= \beta\frac{2}{x^{\beta+1}(1-x)^{\beta+1}}  + \beta(\beta+1) \frac{(2x-1)^2 }{x^{\beta+2}(1-x)^{\beta+2}}  = \beta\frac{2x(1-x) +(1-2x)^2 +\beta (1-2x)^2}{x^{\beta+2}(1-x)^{\beta+2}} \\
            		 &= \beta\frac{(1-x)^2 +x^2 +\beta (1-2x)^2}{x^{\beta+2}(1-x)^{\beta+2}} > 0.
            \end{align*}
Hence $\Phi$ attains its unique minimum at $x=\frac12$, where $\Phi(\frac12)=0$.

Next choose  
\begin{align}
   s>1 \quad \Longleftrightarrow\quad  2^{1-3 s}<2^{-s-1},
    \label{eq:choice_s}\tag{C3}
\end{align}
so that the interval $\big(2^{1-3 s},\,2^{-s-1}\big)$ is nonempty so that later we can find a number $\delta$ between $2^{1-3s}$ and $2^{-s-1}$.

Let $a_{jk}$ denote the midpoint of the interval $I_{jk}$, so that
\[
I_{jk}
=
\left(
a_{jk}-\frac12\delta^{j-1}(1-2\delta),\,
a_{jk}+\frac12\delta^{j-1}(1-2\delta)
\right).
\]
On each interval $I_{jk}$ we place a vertically rescaled copy of $\Phi$, shifted upward by $2^{sj}$. More precisely, for $x\in I_{jk}$ we set
\begin{align*}
    \Phi_{jk}(x) := 2^{sj} + 2^{sj}\Phi\left(\frac{x-a_{jk}}{(1-2\delta)\delta^{j-1}}+\frac12 \right), 
\end{align*}
and we define
\begin{align*}
    w_\delta(x):= \begin{cases}
        \Phi_{jk}(x), & \text{if } x\in I_{jk}\text{ for some } j\ge1,\ k=1,\dots,2^{j-1},\\[0.3em]
        +\infty, & \text{if } x\in \mathbb D_\delta.
    \end{cases}
\end{align*}
The term $2^{sj}$ is an additional level-dependent background slope. It will later be crucial in the curvature estimates, since it enters the denominator of the relevant expressions.

By construction, $w_\delta$ is continuous as a function with values in $[0,\infty]$. On each interval $I_{jk}$ this is immediate. At the endpoints of $I_{jk}$ it follows from the fact that $\Phi(t)\to+\infty$ as $t\searrow0$ or $t\nearrow1$. Finally, if $x\in\mathbb D_\delta$ is not an endpoint of any removed interval, then every sequence approaching $x$ from $[0,1]\setminus\mathbb D_\delta$ must eventually lie in intervals $I_{jk}$ with $j\to\infty$, and hence
$
w_\delta\ge 2^{sj}\to+\infty.
$

Since $\mathcal L^1(\mathbb D_\delta)=0$, only the intervals $I_{jk}$ contribute to the integral. Using $|I_{jk}|=\delta^{j-1}(1-2\delta)$ and $\int_0^1 \Phi(x) \diff x=1$, we obtain
\begin{align*}
    \int_{I_{jk}} w_\delta(x) \diff x
    &
    = \int_{I_{jk}} \Phi_{jk}(x) \diff x 
    = 2^{sj}(1-2\delta)\delta^{j-1}
      +2^{sj}(1-2\delta)\delta^{j-1}\int_0^1 \Phi(t) \diff t \\
    &= 2(1-2\delta)\delta^{j-1}2^{sj}.
\\
 \Rightarrow \quad    \int_0^1 w_\delta(x) \diff x
    &= \sum_{j=1}^\infty \sum_{k=1}^{2^{j-1}} \int_{I_{jk}} w_\delta(x) \diff x 
    = \sum_{j=1}^\infty 2^{j-1}2(1-2\delta)\delta^{j-1}2^{sj} \\
    &= \frac{1-2\delta}{\delta}\sum_{j=1}^\infty (2\delta)^j2^{sj}.
\end{align*}
If, in addition, we assume
\begin{align}
    \delta<2^{-s-1}
    \quad\Longleftrightarrow\quad
    2^{s+1}\delta<1,  \label{eq:condition}\tag{C4}
\end{align}
then the geometric series converges, and
\begin{align*}
    \int_0^1 w_\delta(x) \diff x
    &= \frac{1-2\delta}{\delta}
       \left(
           \frac{2^{s+1}\delta}{1-2^{s+1}\delta}
       \right) 
    = 2^{s+1}\frac{1-2\delta}{1-2^{s+1}\delta}<\infty.
\end{align*}
Hence $w_\delta$ is integrable on $(0,1)$. In particular, it represents an element of $L^1((0,1))$, and it takes the value $+\infty$ on the null set $\mathbb D_\delta$.            
            
\medskip            
\noindent\cb{4} 
Since $w_\delta:[0,1]\to[0,\infty]$ is continuous and $\Psi_2:[0,\infty]\to[0,M]$ is continuous, we have
$
\Psi_2\circ w_\delta\in C^0([0,1]).
$
In particular, $0\le\Psi_2\circ w_\delta\le M$, and therefore
\[
\Psi_2\circ w_\delta\in L^\infty((0,1))\subset L^2((0,1)).
\]
For $x\in I_{jk}$, it holds
$
w_\delta(x)=\Phi_{jk}(x),
$
and hence $\Psi_2\circ w_\delta$ is of class $C^1$ on each interval $I_{jk}$, with
\begin{align}
\begin{aligned}
            	(\Psi_2\circ w_\delta)'(x) &= (\Psi_2\circ \Phi_{jk})'(x)=\Psi_2'\big(\Phi_{jk}(x)\big) \cdot \Phi'_{jk}(x)\\
            	 &= \frac{2^{sj}}{(1+\Phi_{jk}(x) ^2)^{5/4}}\cdot \Phi' \left(  \frac{x-a_{jk}}{(1-2\delta)\delta^{j-1}}  + \frac{1}{2}\right) \cdot \frac{1}{(1-2\delta)\delta^{j-1}}.
            \end{aligned}
            \label{eq:psi_abl}
            \end{align}

We now define a candidate weak derivative for $\Psi_2\circ w_\delta$ on all of $(0,1)$ by
\[
G_\delta(x):=
\begin{cases}
(\Psi_2\circ w_\delta)'(x), & x\in [0,1]\setminus\mathbb D_\delta,\\
0, & x\in \mathbb D_\delta.
\end{cases}
\]
In Step \cbm 5 we will prove that $G_\delta\in L^2((0,1))$, and hence also $G_\delta\in L^1((0,1))$.
Assuming this for the moment, we now show that $\Psi_2\circ w_\delta$ is a pointwise absolutely continuous function and $G_\delta$ is indeed the weak derivative of $\Psi_2\circ w_\delta$, i.e.
\begin{align}
\Psi_2\circ w_\delta(x)
= M+\int_0^x G_\delta(\xi) \diff \xi
\qquad\text{for all }x\in[0,1],
\label{eq:hauptsatz}
\end{align}
where $M=\Psi_2(\infty)$. Note that
$
\Psi_2\circ w_\delta(0)=\Psi_2(w_\delta(0))=\Psi_2(\infty)=M,
$
since $0\in\mathbb D_\delta$.

Let $I_{jk}=(\alpha_{jk},\beta_{jk})$. Since $\alpha_{jk},\beta_{jk}\in\mathbb D_\delta$, we have
$
\Psi_2\circ w_\delta(\alpha_{jk})=\Psi_2\circ w_\delta(\beta_{jk})=M.
$
Therefore,
\begin{align*}
\int_{I_{jk}} G_\delta(\xi) \diff \xi
&=
\lim_{\varepsilon\searrow 0}
\int_{\alpha_{jk}+\varepsilon}^{\beta_{jk}-\varepsilon} (\Psi_2\circ w_\delta)'(\xi) \diff \xi \\
&=
\lim_{\varepsilon\searrow 0}
\bigl(\Psi_2\circ w_\delta(\beta_{jk}-\varepsilon)-\Psi_2\circ w_\delta(\alpha_{jk}+\varepsilon)\bigr)
= M-M=0.
\end{align*}

Now let $x\in[0,1]$.
If $x\in\mathbb D_\delta$, then $\Psi_2\circ w_\delta(x)=M$, and $(0,x)\setminus\mathbb D_\delta$ is the disjoint union of all removed intervals contained in $(0,x)$. Hence, since $G_\delta(x)=0$ for all $x\in \mathbb D_\delta$, it follows 
\begin{align*}
M+\int_0^x G_\delta(\xi) \diff \xi
=
M+\sum_{I_{jk}\subset (0,x)}
\int_{I_{jk}} G_\delta(\xi) \diff \xi
=M=\Psi_2\circ w_\delta(x), \quad \checkmark
\end{align*}
which proves \eqref{eq:hauptsatz} in this case.

Now suppose that $x\notin \mathbb D_\delta$. Then there exists a unique interval
$
I_{j_0k_0}=(\alpha_{j_0k_0},\beta_{j_0k_0})
$
such that $x\in I_{j_0k_0}$. Since $\alpha_{j_0k_0}\in\mathbb D_\delta$, the previous case gives
\[
\int_0^{\alpha_{j_0k_0}} G_\delta(\xi) \diff \xi =0.
\]
This yields \eqref{eq:hauptsatz} also for $x\notin\mathbb D_\delta$ since
\begin{align*}
 \int_0^x  G_\delta(\xi) \diff \xi +M  
                &=  \int_0^{\alpha_{j_0k_0}}  G_\delta(\xi) \diff \xi +
                \int_{\alpha_{j_0k_0}}^x G_\delta(\xi) \diff \xi +M
                \\
&= 0+
\lim_{\varepsilon\searrow 0}
\int_{\alpha_{j_0k_0}+\varepsilon}^x (\Psi_2\circ w_\delta)'(\xi) \diff \xi +M\\
&=
\lim_{\varepsilon\searrow 0}
\bigl(\Psi_2\circ w_\delta(x)-\Psi_2\circ w_\delta(\alpha_{j_0k_0}+\varepsilon)\bigr)+M
=
\Psi_2\circ w_\delta(x). \quad \checkmark
\end{align*}
Hence, once Step \cbm 5 establishes $G_\delta\in L^2((0,1))$, we conclude that
\[
\Psi_2\circ w_\delta\in W^{1,2}((0,1)),
\qquad
(\Psi_2\circ w_\delta)'=G_\delta
\quad\text{a.e. in }(0,1).
\]

            \medskip 
    \noindent\cb{5}
We now estimate the local curvature contribution on each removed interval $I_{jk}$.
Let
\[
E_{jk}:=\int_{I_{jk}} \big|(\Psi_2\circ w_\delta)'(x)\big|^2 \diff x .
\]
In what follows, we write $A\preceq B$ if $A\le C B$ for some constant $C>0$ independent of $j$ and $k$.  It is important to observe that the following estimates are uniform with respect to $k$ and $j$. Using the formula \eqref{eq:psi_abl} from Step \cbm 4 and a change of variables we obtain
\begin{align*}
   	E_{jk} & = \frac{\delta^2}{(1-2\delta)^2 \delta^{2j}} \int_{I_{jk}} \frac{\left[2^{sj}\Phi' \left( \frac{x-a_{jk}}{(1-2\delta)\delta^{j-1}}+\frac{1}{2}\right)\right]^2}{\left( 1+ \left[2^{sj}+2^{sj} \Phi\left( \frac{x-a_{jk}}{(1-2\delta)\delta^{j-1}}+\frac{1}{2}\right)\right]^2\right)^{5/2}} \diff x\\
   	&\preceq \frac{\delta}{(1-2\delta) \delta^{j}} \int_{0}^1 \frac{\left[2^{sj}\Phi' \left(y\right)\right]^2}{\left( 1+ \left[2^{sj}+ 2^{sj}\Phi\left( y\right)\right]^2\right)^{5/2}} \diff y
   	\underset{\eqref{eq:condition}}{\overset{2\delta<2^{-s}}\preceq} \delta ^{-j} \int_0^1 \frac{\left[2^{sj}\Phi' \left(y\right)\right]^2}{ \left(2^{sj}+ 2^{sj} \Phi\left( y\right)\right)^{5}} \diff y 
            \end{align*}
 Consequently, since the integrand is invariant under the substitution, it follows that
\begin{align*}
   	E_{jk} 
   	&\preceq \delta^{-j} 2^{sj(2-5)}\int_0^1 \frac{ y^{-2(\beta+1)}(1-y)^{-2(\beta+1)}}{(1+ y^{-\beta}(1-y)^{-\beta})^5}\diff y
   	\preceq 2^{-3sj}\delta^{-j} \int_0^1 \frac{ y^{-2(\beta+1)}}{(1+ y^{-\beta})^5}\diff y \\
  &\preceq  2^{-3sj}\delta^{-j}\int_0^{1}
\frac{y^{-2(\beta+1)}}{y^{-5\beta}}\,\diff y
= 2^{-3sj} \delta^{-j}\int_0^{1} y^{-2+3\beta}\,\diff y .
            \end{align*}
Since $\beta>\frac13$, the exponent satisfies
\[
-2+3\beta>-1,
\]
and therefore the last integral is finite.
We conclude that
\[
E_{jk}\preceq \delta^{-j}2^{-3 sj}
= \bigl(\delta^{-1}2^{-3s}\bigr)^j .
\]
In particular, this estimate is uniform with respect to $k$.
Summing over all removed intervals, we obtain
 \begin{align*}
            \int_{[0,1]\setminus \mathbb D_\delta} \big[ (\Psi_2\circ w_\delta)'(x)\big]^2 \diff x 
            =	\sum_{j=1}^\infty \sum_{k=1}^{2^{j-1}} E_{jk} \preceq \sum_{j=1}^\infty \sum_{k=1}^{2^{j-1}}  \big(\delta^{-1} 2^{-3s}\big)^j 
            	\le \sum_{j=1}^\infty   \big(\delta^{-1} 2^{1-3s}\big)^j < \infty
            \end{align*}
This series converges provided
\begin{align}
\delta^{-1}2^{1-3 s}<1
\quad\Longleftrightarrow\quad
\delta>2^{1-3 s}.
\label{eq:condition2}\tag{C5}
\end{align}
Together with \eqref{eq:condition}, this means that we may choose
\[
2^{1-3 s}<\delta<2^{-s-1}.
\]
Since $s>1$, this interval is nonempty, so that such a choice of $\delta$ is an admissible choice. Thus $(\Psi_2\circ w_\delta)' \in L^2\bigl((0,1)\setminus\mathbb D_\delta\bigr)$, and since $\mathbb D_\delta$ has Lebesgue measure zero, Step \cbm 4 yields
\[
\Psi_2\circ w_\delta \in W^{1,2}((0,1)).
\]
By the discussion in Sect.~\ref{sec:preliminaries}, $w_\delta\in C^0([0,1];[0,+\infty])$ with respect to the topology of the extended half-line.  The crucial point is the additional term $2^{sj}$, stemming from the structure $|u''|^2 (1+u'^2)^{-5/2}$  in the definition of $\Phi_{jk}$. It yields the decay factor $2^{-3sj}$ ensuring convergence of the sum over $j$.
 \medskip

\noindent\cb{6}
We now fix parameters $\beta\in\left(\frac13,\frac{1}{2}\right)$, $s>1$, and $\delta\in\left(2^{1-3 s},2^{-s-1}\right)$, i.e. conditions \eqref{eq:1dcond1} --\eqref{eq:condition2}  are satisfied. We then define the function from the statement of this theorem
\[
u(x):=U(x)+f_\delta(x), \quad\text{ where }\quad U(x):=\int_0^x w_\delta(\xi) \diff \xi,
\]
and $f_\delta$ is the generalized Cantor function from Step \cbm 2. Since
$w_\delta\in L^1((0,1))$ by Step \cbm 3, the above integral is well defined. Hence
\[
U\in W^{1,1}((0,1))\subset BV((0,1)),
\]
so in $U$ is absolutely continuous on $[0,1]$. Since also
$f_\delta\in BV((0,1))\cap C^0([0,1])$, it follows that
\[
u\in BV((0,1))\cap C^0([0,1])
\quad \text{ with }\quad
Du = w_\delta\,\mathcal L^1 + D^cf_\delta .
\]
Equivalently,  the absolutely continuous density of $Du$ is $w_\delta$, its Cantor part coincides with that of $f_\delta'$, and its jump part vanishes:
\[
(u')^a = w_\delta \quad\text{a.e. on }(0,1),\qquad
D^c u = D^c f_\delta,\qquad
D^ju = 0.
\]
In fact, with respect to the chosen continuous representative \(w_\delta\), we have the generalized Cantor set 
\begin{align*}
         	\mathbb D_\delta = \bigg\{ x\in [0,1] \ \bigg| \ \Big| (u')^a(x)\Big| = \infty \bigg\} 
          	 = \bigg\{ x\in [0,1] \ \bigg| \ (u')^a(x) = +\infty \bigg\} 
            \end{align*}
is closed and, in particular, of measure zero. Since $f_\delta$ is continuous, its singular part is purely Cantor. Moreover,
\[
|D^c u|((0,1))
=
|D^c f_\delta|((0,1))
=
f_\delta(1)-f_\delta(0)
=
1>0.
\]
In particular, $u\notin SBV((0,1))$.

It remains to clarify the role of $U$. Although $U$ provides the absolutely continuous part needed for finite curvature energy, it does not belong to $W^{2,1}((0,1))$. Indeed, on each open interval $I_{jk}$, the function $U$ is of class $C^2$, and
\[
U''(x)=w_\delta'(x)
=
\frac{2^{sj}}{(1-2\delta)\delta^{j-1}}
\Phi'\!\left(
\frac{x-a_{jk}}{(1-2\delta)\delta^{j-1}}+\frac12
\right),
\qquad x\in I_{jk}.
\]
Hence $U''\in L^1_{\mathrm{loc}}(I_{jk})$ for every $k,j$. However, with the same change of variables as in Step~\cbm 3
we obtain
\begin{align*}
2^{-sj}\int_{I_{jk}} |U''(x)| \diff x
&= \int_0^1 |\Phi'(y)| \diff y 
= C_\beta\beta \int_0^1
\frac{|2y-1|}{y^{\beta+1}(1-y)^{\beta+1}} \diff y \\
   	&\   \succeq \int_{0}^{1/4}  \frac{1 }{y^{\beta+1}} \diff y +\int_{3/4}^1  \frac{1 }{(1-y)^{\beta+1}} \diff y
            \succeq \int_{0}^{1/4}  \frac{1 }{y^{4/3}} \diff y=\infty,
\end{align*}
since $\beta>1/3$.
Thus $U''\notin L^1(I_{jk})$ for every $k,j$, and therefore
\[
U\notin W^{2,1}((0,1)).
\]

Nevertheless, Step \cbm 5 shows that $\Psi_2\circ w_\delta \in W^{1,2}((0,1))$.
Consequently, the graph of $U$ has finite curvature energy:
       \begin{align*}
                \int_0^1 \kappa_U(x)^2 \sqrt{1+U'(x)^2}\diff x 
            	= \int_{\operatorname{graph}[U]} \kappa_U(x)^2\diff s_U(x)  = 
            	\int_0^1 \Big( \frac{\diff\ }{\diff x} \Psi_2\big(\underbrace{ U'(x)}_{=w_\delta(x)}\big)\Big)^2\diff x <\infty.
            \end{align*}    
where
$
\kappa_U=U''(1+U'^2)^{-3/2}
$
denotes the classical curvature on $(0,1)\setminus \mathbb D_\delta$. Since
\[
\mathcal H^1(\operatorname{graph}(U))
\le 1+\|U'\|_{L^1((0,1))}<\infty,
\]
it follows by Cauchy-Schwarz that $\kappa_U\in L^1(\operatorname{graph}[U])$ as well.

Thus $u=U+f_\delta$ already has the required nontrivial Cantor part, while $U$ provides the absolutely continuous contribution with finite curvature energy. The finiteness of the relaxed Willmore energy of $u$ will be established in the next step.

\medskip

\noindent\cb{7}
We now prove that
\[
\overline{\W}(u)\le \int_0^1 \left|(\Psi_2\circ w_\delta)'(x)\right|^2 \diff x
=\W(U)=:\W^a(u).
\]
\begin{proof}
By the relaxation theorem~\ref{thm:maso}, there is a recovery sequence $q_n\in W^{2,2}((0,1))$ such that $q_n\to u$ in $L^1((0,1))$ and
\[
\mathcal F_2(q_n)\longrightarrow
|Du|((0,1))+\int_0^1 \big|(\Psi_2\circ w_\delta)'(x)\big|^2 \diff x.
\]
The lower semicontinuity of the total variation gives
\[
|Du|((0,1))\le\liminf_{n\to\infty}\int_0^1|q_n'|\,\diff x.
\]
Since $\mathcal F_p(q_n)=\int_0^1|q_n'|\diff x+\W(q_n)$, it follows that
\begin{align*}
\limsup_{n\to\infty}\mathcal \W(q_n)
&\le \lim_{n\to\infty}\mathcal F_2(q_n)
- \liminf_{n\to\infty}\int_0^1|q_n'|\,\diff x
\le \int_0^1 \big|(\Psi_2\circ w_\delta)'(x)\big|^2 \diff x.
\end{align*}
Taking the infimum over all approximating sequences proves the claim.
\end{proof}
  
Therefore, $\overline{\W}(u)\le\W^a(u) <\infty$. Together with $|D^c u|((0,1))>0$, which was shown in Step \cbm 6, this proves the theorem.
\end{proof}

\begin{corollary}
[Arbitrarily small relaxed Willmore energy with a nontrivial Cantor part]
\label{cor:no-positive-gap-pure-Willmore}
For every $\varepsilon>0$, there exists a continuous and nondecreasing
function
\[
u_\varepsilon\in BV((0,1))\setminus SBV((0,1))
\]
such that
\[
|D^c u_\varepsilon|((0,1))=1
\qquad\text{and}\qquad
\overline{\W}(u_\varepsilon)<\varepsilon.
\]
In particular,
\[
\inf\left\{
\overline{\W}(v)
\,\middle|\,
v\in BV((0,1)),\
|D^c v|((0,1))=1
\right\}
=0.
\]
\end{corollary}

\begin{proof}
Let $u=U+f_\delta$ be the function constructed in Thm.~\ref{thm:finit_will}, and let $w_\delta=U'$ denote the absolutely continuous part of its derivative.
On every removed interval of generation $j$, the construction gives
\[
w_\delta\geq 2^{sj}
\quad \Rightarrow \quad\forall j\geq1 \colon\quad w_\delta(x)\geq c_0:=2^s>0
\qquad\text{for a.e. }x\in(0,1).
\]
For $A\ge1$, define
\[
u_A:=A\,U+f_\delta
\quad \Rightarrow \quad 
Du_A=A w_\delta\,\mathcal L^1+D^c f_\delta .
\]
Hence $D^c u_A=D^c f_\delta\neq0$, and therefore $u_A\notin SBV((0,1))$ for every $A\geq1$.

The blow-up set of the continuous $\overline{\mathbb R}$-valued representative of $A w_\delta$ is still $\mathbb D_\delta$, and the positive singular measure $D^s u_A=D^c f_\delta$ is concentrated on this set. Moreover, for a.e. $x\in(0,1)\setminus\mathbb D_\delta$,
\[
\left|
\bigl(\Psi_2\circ(Aw_\delta)\bigr)'(x)
\right|^2
=
\frac{A^2|w_\delta'(x)|^2}
     {\bigl(1+A^2w_\delta(x)^2\bigr)^{5/2}}.
\]
Using $A\geq1$ and $w_\delta\geq c_0$, we further obtain $w_\delta\sqrt{1+c_0^{-2}} \ge \sqrt{w^2_\delta+1}$
\[
\frac1{w_\delta^5}
\le
\frac{\left(1+c_0^{-2}\right)^{5/2}}{(1+w_\delta^2)^{5/2}}
\quad\Rightarrow\quad
\frac{A^2|w_\delta'|^2}{(1+A^2w_\delta^2)^{5/2}}
\le
\frac{|w_\delta'|^2}{A^3w_\delta^5}
\le \frac{\left(1+c_0^{-2}\right)^{5/2}|w_\delta'|^2}{A^3(1+w_\delta^2)^{5/2}}.
\]
Therefore, it follows
\[\int_0^1
\frac{A^2|w_\delta'(x)|^2}{(1+A^2 w_\delta(x)^2)^{5/2}} \diff x 
\le
A^{-3}
\left(1+c_0^{-2}\right)^{5/2}
\int_0^1
\frac{|w_\delta'(x)|^2}{(1+w_\delta(x)^2)^{5/2}} \diff x .
\]
The final expression is integrable by the construction in Theorem~\ref{thm:finit_will}. Repeating the  arguments from Steps~\cbm 4 and \cbm 5 gives
\[
\Psi_2\circ(Aw_\delta)\in W^{1,2}((0,1)).
\]
Hence $u_A\in X_\psi^2((0,1))$. Step~\cbm7 from  Thm.~\ref{thm:finit_will} therefore yields
\[
\begin{aligned}
\overline{\W}(u_A)
&\leq
\int_0^1
\frac{A^2|w_\delta'(x)|^2}
     {\bigl(1+A^2w_\delta(x)^2\bigr)^{5/2}}
\,\diff x\\
&\leq
\frac{(1+c_0^{-2})^{5/2}}{A^3}
\int_0^1
\frac{|w_\delta'(x)|^2}
     {\bigl(1+w_\delta(x)^2\bigr)^{5/2}}
\,\diff x.
\end{aligned}
\]
Thus there exists a constant $C>0$, independent of $A$, such that
\[
\overline{\W}(u_A)\leq C A^{-3}.
\]
Choose $A>\max\{1,{C}^{1/3}\varepsilon^{-1/3}\}$ and set $u_\varepsilon:=u_A$. Then
\[
|D^c u_\varepsilon|((0,1))=1
\qquad\text{and}\qquad
\overline{\W}(u_\varepsilon)<\varepsilon.
\]

\end{proof}


\begin{remark}[No positive energy gap for the pure relaxed Willmore term]
\label{rem:no-positive-gap-pure-Willmore}
Geometrically, the Cantor part is kept fixed, while only the absolutely continuous part of the graph is stretched in the vertical direction. The regular graph becomes steeper and longer, but its bending contribution decreases like $A^{-3}$.

This observation concerns only the pure relaxed Willmore term $\overline{\W}$. It is not a corresponding small-energy statement for the full Dal Maso-Fonseca-Leoni-Morini functional $\mathcal F_2$. Indeed, for smooth functions this functional contains, in addition to the Willmore term, the total variation term. For the rescaled functions $u_A=A\,U+f_\delta$, the total variation satisfies \[ |Du_A|((0,1)) = A\int_0^1 w(x) \diff x + |D^c f_\delta|((0,1)). \] Thus the total variation grows linearly in $A$, although the pure Willmore term decays like $A^{-3}$. This is precisely why the above argument applies to $\overline{\W}$, but not to the full relaxed functional $\overline{\mathcal F}_2$.

This remark rules out one natural strategy. In the parametric Willmore theory, energy thresholds and small-energy regularity results are decisive tools, see, for instance, the small-energy theory of Kuwert and Sch\"atzle \cite{kuwert2001willmore}. For the pure relaxed Willmore term of graphs, no analogous mechanism can exist: there is no positive energy threshold below which Cantor parts are excluded. Consequently, any regularity theory for minimizers of relaxed Willmore-type graph functionals must rely on minimality or on the lower-order terms, and not on smallness of the curvature energy alone. One idea is to consider fixed boundary data with smallness assumptions, as used for elliptic and parabolic Willmore problems, see \cite{gulyak2026willmore, gulyak2026}. 
\end{remark}

\begin{remark}[Cantor part vanishes by rotation] \label{rem:rot}
It is important to note that the appearance of a Cantor part in the derivative is tied to the chosen projection direction. Let $g:[0,1]\to\mathbb R$ be continuous and nondecreasing, and rotate its graph clockwise by an angle $\varphi\in(0,\frac{\pi}{2})$. Writing
\[ X_\varphi(x):=x\cos\varphi+g(x)\sin\varphi, \qquad Y_\varphi(x):=-x \sin\varphi+ g(x)\cos\varphi,\]
we see that $X_\varphi$ is strictly increasing. Hence the rotated curve can again be written as a graph, say
\[
Y_\varphi=v_\varphi(X_\varphi).
\]
Moreover, for $x_1<x_2$, by expoiling the monotonicity of $f$
\[ 
|Y_\varphi(x_2)-Y_\varphi(x_1)| \le \sin\varphi\,|x_2-x_1|+\cos\varphi\,|g(x_2)-g(x_1)| \le L_\varphi |X_\varphi(x_2)-X_\varphi(x_1)|,
\]
where one may take
\[
L_\varphi:=\tan\varphi+\cot\varphi.
\]
Thus $v_\varphi$ is Lipschitz. By \cite[Theorem 3.16, p.~127]{ambrosio2000functions}, $v_\varphi\in W^{1,\infty}$, and therefore its distributional derivative has no singular part.

Applied to the Cantor function $f_\delta$, this shows that the Cantor part of $D f_\delta$ is not invariant under changes of projection: it disappears after such a rotation.  The same projection effect occurs for the function $u$ constructed in Thm.~\ref{thm:finit_will}. Since $u$ is continuous and nondecreasing, every clockwise rotation by an angle $\varphi\in(0,\frac{\pi}{2})$ yields a Lipschitz graph, whose distributional derivative has no singular part. In the original projection, however,
\[
D^c u=D^c f_\delta\neq0.
\]
The crucial difference from the pure Cantor function is that this Cantor part is concentrated precisely where the absolutely continuous slope blows up, namely on $Z^+[(u')^a]$. Thus the Cantor part is compatible with finite relaxed Willmore energy: it is a projection-induced singular part and creates no additional curvature cost in the approximation constructed above.
\end{remark}

\subsection*{Plot with explicit values}

We illustrate the function $u$ from Thm.~\ref{thm:finit_will} for the
parameter choice
$ s=2>1$, $\delta=\frac{1}{31}$, and $\beta=\frac{2}{5}\in\left(\frac13,\frac12\right)$.
First, we check that all assumptions are satisfied. Indeed,  the conditions \eqref{eq:condition}, \eqref{eq:condition2} are satisfied:
\[
2^{1-3s}
= 2^{-5}
= \frac{1}{32}
< \frac{1}{31}
= \delta
< \frac{1}{8}
= 2^{-s-1}.
\]

\begin{figure}[H]
\centering
\includegraphics{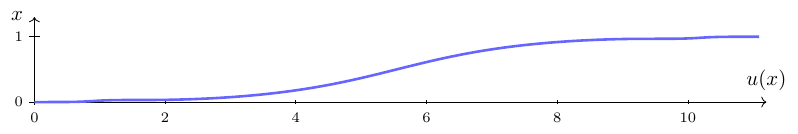}
\caption{The function $u$ constructed in Thm.~\ref{thm:finit_will} with $s=2$, $\delta=\frac{1}{31}$, and $\beta=\frac{2}{5}$. The axes are interchanged: the value $u(x)$ is plotted on the horizontal axis and $x$ on the vertical axis, so that both axes can be shown with the same scale.}
\label{fig:explicit-plot-u}
\end{figure}

Since $u$ is nondecreasing, its maximum is attained at $x=1$. We have
\[
u(1)=U(1)+f_\delta(1).
\]
By the normalization of the Cantor function, $f_\delta(1)=1$, while
\[
u(1)=U(1)+1
=
\int_0^1 w_\delta(x) \diff x +1
=
2^{s+1}\frac{1-2\delta}{1-2^{s+1}\delta}+1 
=
11+\frac{2}{23}
\approx 11.09 .
\]

\subsection*{The graph as a \texorpdfstring{$C^1$}{C1}-curve}
In this subsection we show that the graph of the function constructed in Thm.~\ref{thm:finit_will} can be represented as the image of a regular $C^1$-curve in $\mathbb R^2$, parametrized by arc-length. More precisely, although the derivative of $u$ has a nontrivial Cantor part, the corresponding graph admits a continuous unit tangent field, which becomes vertical precisely on the Cantor set. Since $u$ has no jump part, this graph coincides with the usual graph of $u$.

As in \cite{acerbi2017geometric}, $c_u$ denotes the Cartesian curve
associated with $u\in BV((0,1))\cap C^0([0,1])$, namely
\[
c_u:[0,1]\to\mathbb R^2,
\qquad
c_u(x):=(x,u(x)).
\]
Since $u$ has no jump part, $c_u$ is continuous and belongs to
$BV((0,1);\mathbb R^2)\cap C^0([0,1];\mathbb R^2)$. Writing
\[
Du=(u')^a\,\mathcal L^1 + D^c u,
\]
we have, in the sense of vector-valued Radon measures; see
\cite[Remark~3.1, p.~459]{acerbi2017geometric},
\begin{align}
Dc_u &= (1,(u')^a)\,\mathcal L^1 + (0,1)\,D^c u .
\end{align}
Since $\mathcal L^1$ and $|D^c u|$ are mutually singular, the total variation
of $Dc_u$ splits accordingly, and hence
\begin{align}
|Dc_u|
=
\sqrt{1+|(u')^a|^2}\,\mathcal L^1 + |D^c u|.
\label{eq:Dcu}
\end{align}

We define the \emph{arc-length function}
\[
s:[0,1]\to[0,L],
\qquad s(x):=|Dc_u|([0,x]),
\qquad \text{ where }\quad L:=|Dc_u|([0,1]).
\]
The measure $\mathcal L^1$ is atomless, and $|D^c u|$ is atomless because it is the Cantor part of $Du$, see, for instance, \cite[p.~139]{ambrosio2000functions}. Therefore $|Dc_u|$ has no atoms, and the distribution function $s$ is continuous on $[0,1]$. Moreover, if $0\le x<y\le 1$, then by \eqref{eq:Dcu}
\[
s(y)-s(x) = |Dc_u|((x,y]) 
\ge
\int_x^y \sqrt{1+|(u')^a(t)|^2} \diff t
\ge
\int_x^y 1 \diff t
= y-x > 0.
\]
Thus $s$ is strictly increasing. Consequently,
$s:[0,1]\to[0,L]$ is a homeomorphism, and its inverse $s^{-1}:[0,L]\to[0,1]$ is well defined.

We can now introduce the arc-length parametrization of the graph of the function constructed in Thm.~\ref{thm:finit_will} by setting
\[
\gamma:[0,L]\to\mathbb R^2,
\qquad
\gamma(\sigma):=c_u(s^{-1}(\sigma)).
\]
The following result shows that, in the present construction, this parametrized Cartesian curve is in fact a regular $C^1$-curve. This is in agreement with the general arc-length parametrization theory for continuous Cartesian curves developed by Acerbi and Mucci, see \cite[Prop.~5.5 and Thm.~10.2 and the discussion after formula~(10.5)]{acerbi2017geometric}.  In their framework such curves arise naturally in the study of the geometric functionals $\mathcal F_p$, in particular in the case $p=1$. For the reader's convenience, and also to make the special structure of the present one-dimensional construction explicit, we give a direct proof.

\begin{thm}\label{thm:c1curv}
Let $u\colon[0,1]\to\mathbb R$ be the function constructed in Thm.~\ref{thm:finit_will}. In particular, $u\in BV((0,1))\cap C^0([0,1])$, $Du$ has a nontrivial Cantor part, and $ (u')^a=w_\delta $ on $[0,1]\setminus\mathbb D_\delta $, as the chosen continuous representative.
Let
\[
c_u(x):=(x,u(x)),
\qquad
s(x):=|Dc_u|([0,x]),
\qquad
L:=|Dc_u|([0,1]),
\]
and define the arc-length parametrization
\[
\gamma:[0,L]\to\mathbb R^2,
\qquad
\gamma(\ell):=c_u(s^{-1}(\ell)).
\]
Then the following assertions hold.

\begin{itemize}
\item[\cb{1}]
The curve $\gamma$ is an injective regular $C^1$-curve parametrized by
arc-length. Moreover, for every $x\in[0,1]$,
\begin{align*}
\gamma(s(x))=c_u(x),
\qquad \gamma_1(s(x))=x.
\end{align*}
Its unit tangent field, pulled back to the original parameter $x$, is given by
\begin{align}
\gamma'(s(x))
=
\begin{cases}
\dfrac{(1,w_\delta(x))}{\sqrt{1+w_\delta(x)^2}},
& \text{if }x \in [0,1]\setminus\mathbb D_\delta,\\[1.2em]
(0,1),
& \text{if }x\in \mathbb D_\delta .
\end{cases}
\label{eq:tau-def}
\end{align}

\item[\cb{2}]
The curve $\gamma$ has square-integrable curvature. More precisely, $\gamma\in W^{2,2}(0,L;\mathbb R^2)$. If $\kappa_\gamma$ denotes the scalar curvature of the arc-length
parametrized curve $\gamma$, defined by
\[
\gamma''=\kappa_\gamma\,\nu_\gamma
\qquad\text{a.e. in }(0,L),
\]
where $\nu_\gamma$ is the unit normal obtained by rotating $\gamma'$ by
$\pi/2$, then $\kappa_\gamma\in L^2(0,L)$ and
\[
\kappa_\gamma(\ell)=
\begin{cases}
\dfrac{w_\delta'(s^{-1}(\ell))}{(1+w_\delta(s^{-1}(\ell))^2)^{3/2}},
& \ell\in s\left([0,1]\setminus\mathbb D_\delta\right),\\
0,
& \ell\in s(\mathbb D_\delta).
\end{cases} 
\quad\text{for a.e. }\ell\in(0,L)
\]
Consequently,
\begin{align}
\int_0^L |\kappa_{\gamma}(\ell)|^2 \diff\ell
&=
\int_0^1
\frac{|w_\delta'(x)|^2}
{\bigl(1+w_\delta(x)^2\bigr)^{5/2}} \diff x
=
\W^a(u)
<\infty .
\label{eq:kappa-L2}
\end{align}

\item[\cb{3}]
The graph $c_u([0,1])$ is not a one-dimensional $C^2$-submanifold of $\mathbb R^2$. More precisely, for every $x_0\in\mathbb D_\delta$, the set $c_u([0,1])$ does not admit a local regular $C^2$-parametrization in any neighbourhood of $c_u(x_0)$.
\end{itemize}
\end{thm}

\begin{proof}
\cb{1} \textbf{$C^1$-arc-length parametrization.} \\ 
Since $u$ is monotonically increasing in the construction of Thm.~\ref{thm:finit_will}, its Cantor part is a positive measure. Hence $|D^c u|=D^c u $. By the polar decomposition theorem for vector-valued Radon measures \cite[Corollary~1.29]{ambrosio2000functions}, there exists a unique  $|Dc_u|$-measurable function $\tau\in [L^1([0,1],|Dc_u|)]^2$ such that
\[
\tau:[0,1]\to\mathbb S^1, \qquad Dc_u=\tau\,|Dc_u|.
\]
We now characterize this polar vector explicitly.
On the absolutely continuous part we have $(u')^a=w_\delta$. Then \eqref{eq:Dcu}  gives 
\[
(1,(u')^a)\mathcal L^1 = \tau \sqrt{1+|(u')^a|^2}\mathcal L^1
\quad\Rightarrow\quad \tau(x)
=
\frac{(1,w_\delta(x))}{\sqrt{1+w_\delta(x)^2}}
\qquad
\mathcal L^1\text{-a.e. on }[0,1]\setminus\mathbb D_\delta .
\]
On the Cantor part, by \cite[Remark~3.1, p.~459]{acerbi2017geometric} the measure $Dc_u$ equals $(0,1)D^c u$, whereas
$ |Dc_u|=D^c u$. Hence 
\[
\tau(x)=(0,1)
\qquad
D^c u\text{-a.e. on }\mathbb D_\delta .
\]
We now choose a continuous representative of this measure-theoretic tangent.
Consider the map
\[
F:[0,\infty]\to\mathbb S^1,
\qquad
F(t):=
\begin{cases}
\dfrac{(1,t)}{\sqrt{1+t^2}},
& t<\infty,\\[0.5em]
(0,1),
& t=\infty .
\end{cases}
\]
Here $[0,\infty]$ is endowed with the usual topology of the extended half-line. Since $w_\delta\colon[0,1]\to[0,\infty]$ is continuous and $w_\delta(x)=+\infty$ precisely on $\mathbb D_\delta$, the composition  $F\circ w_\delta$ is continuous on $[0,1]$. Moreover, the preceding identification shows that $F(w_\delta)=\tau$ holds $|Dc_u|$-a.e. Thus, replacing $\tau$ by this continuous representative, we may write
\[
\tau(x)
=
\begin{cases}
\dfrac{(1,w_\delta(x))}{\sqrt{1+w_\delta(x)^2}},
& \text{if }x\in[0,1]\setminus\mathbb D_\delta,\\[0.5em]
(0,1),
& \text{if }x\in\mathbb D_\delta .
\end{cases}
\]
which later will show \eqref{eq:tau-def}.

Now, keeping in mind that $s$ is a homeomorphism, we finally define the curve $\gamma$ by setting
\[
\gamma(\ell)
:=
c_u(0)+\int_0^\ell \tau(s^{-1}(r)) \diff r,
\qquad \ell\in[0,L].
\]
The map $\tau\circ s^{-1}:[0,L]\to\mathbb S^1$ is continuous. Hence
\[
\gamma\in C^1([0,L];\mathbb R^2),
\qquad
\gamma'(\ell)=\tau(s^{-1}(\ell))
\quad\text{for every }\ell\in[0,L].
\]
Since $|\tau|=1$ everywhere, we obtain $|\gamma'(\ell)|=1
$ for every $\ell\in[0,L],$
and therefore $\gamma$ is parametrized by arc-length.

It remains to prove that this parametrization coincides with the Cartesian curve $c_u$ after the change of variables $s$. We first claim that the pushforward measure of $\mathcal L^1$ under $s^{-1}$ is
\begin{align}
(s^{-1})_\#\bigl(\mathcal L^1\mrest[0,L]\bigr)
=
|Dc_u|.
\label{eq:sigmapush}
\end{align}
Indeed, for every $x\in[0,1]$, 
\[
\bigl((s^{-1})_\#(\mathcal L^1\mrest[0,L])\bigr)([0,x])
=
\mathcal L^1\bigl(s([0,x])\bigr)
=
\mathcal L^1([0,s(x)])
=
s(x)
=
|Dc_u|([0,x]).
\]
Since both sides are finite Borel measures on $[0,1]$ and coincide on the generating family of intervals $[0,x]$, \eqref{eq:sigmapush} follows.
Now let $x\in[0,1]$. Using the change-of-variables formula for push-forward measures \eqref{eq:sigmapush} and the polar decomposition
$Dc_u=\tau\,|Dc_u|$, we obtain
\begin{align*}
\gamma(s(x))
&=
c_u(0)+\int_0^{s(x)} \tau(s^{-1}(r)) \diff r  
=
c_u(0)+\int_{[0,x]} \tau \diff |Dc_u| \\
&=
c_u(0)+Dc_u([0,x]).
\end{align*}
Since $c_u\in BV((0,1);\mathbb R^2)\cap C^0([0,1];\mathbb R^2)$, the
one-dimensional structure theorem for BV functions
\cite[Theorem 3.28]{ambrosio2000functions} applied componentwise gives, for $0<y<x<1$,
\[
c_u(x)-c_u(y)=Dc_u((y,x]).
\]
Since $c_u$ is continuous, the measure $Dc_u$ has no atoms. Hence
\[
Dc_u((y,x])=Dc_u([y,x]) \quad \text{ and }
\quad 
c_u(x)-c_u(0)=Dc_u([0,x]).
\]
Consequently, for $\gamma$ and its first component  $\gamma_1$ it holds
\[
\gamma(s(x))=c_u(x)
\quad \text{ and }
\quad \gamma_1(s(x))=x
\qquad\text{for every }x\in[0,1].
\]
Equivalently, since $s$ is onto, $\gamma_1(\ell)=s^{-1}(\ell)$ for every $\ell\in[0,L]$.
In particular, $\gamma$ is injective. Moreover, $\gamma([0,L])=c_u([0,1])$. Thus $\gamma$ is an injective regular $C^1$-curve whose image is precisely the graph of $u$.
Finally, from the definition of $\gamma$ we obtain 
\[ 
\forall\,\ell\in[0,L]\colon \quad \gamma'(\ell)=\tau(s^{-1}(\ell))
\quad \overset{\ell=s(x)}\Rightarrow \quad
\forall\,x\in[0,1]\colon\quad \tau(x)=\gamma'(s(x)).
\]
Using the explicit representation of $\tau$, this gives \eqref{eq:tau-def}.



\medskip 
\noindent\cb{2} \textbf{Square-integrable curvature.}\\
We first show that the arc-length parametrization constructed above has square-integrable curvature. This should be contrasted with the fact proved below that the same graph is not a one-dimensional $C^2$-submanifold of $\mathbb R^2$. Motivated by the representation  of the tangent vector $\gamma'$ in \eqref{eq:tau-def}, where the quotient of the components equals  $w_\delta/1$, we define the \emph{tangent angle with respect to the horizontal parameter} by 
\begin{align*}
\theta(x) = \begin{cases}
\arctan(w_\delta(x)),
& \text{if }x \in [0,1]\setminus\mathbb D_\delta,\\[0.3em]
\pi/2,
& \text{if }x\in \mathbb D_\delta,
\end{cases}
\qquad x\in[0,1].
\end{align*}
Equivalently, $\theta$ is obtained by composing $w_\delta$ with the continuous extension of $\arctan$ to $[0,\infty]$, where $\arctan(+\infty):=\pi/2$. Since $w_\delta:[0,1]\to[0,\infty]$ is continuous, we have $\theta\in C^0([0,1])$.
Moreover, the unit tangent field found in \cbm{1} can be written as
\[
\tau(x)=\bigl(\cos\theta(x),\sin\theta(x)\bigr),
\qquad x\in[0,1].
\]
Thus the value $\theta=\pi/2$ on $\mathbb D_\delta$ is not an additional choice, but exactly the angle corresponding to the vertical polar vector of
$Dc_u$.
On every connected component $I_{jk}$ of $[0,1]\setminus\mathbb D_\delta$, the function $w_\delta$ is $C^1$. Therefore, classically on $I_{jk}$, it follows
\[
\theta'(x)
=
\frac{\diff\ }{\diff x}\arctan(w_\delta(x))
=
\frac{w_\delta'(x)}{1+w_\delta(x)^2}.
\]
On the absolutely continuous part of $|Dc_u|$ we have
$
\diff |Dc_u|
=
\sqrt{1+w_\delta^2} \diff x.
$
\[
\Rightarrow\quad
\theta'(x) \diff x
=
\frac{w_\delta'(x)}{1+w_\delta(x)^2} \diff x
=
\frac{w_\delta'(x)}
{\bigl(1+w_\delta(x)^2\bigr)^{3/2}}
 \diff |Dc_u|, \quad \text{ on } [0,1]\setminus\mathbb D_\delta,
\]

We therefore define the \emph{candidate for the derivative of the tangent angle with
respect to arc-length} by
\[
g(x):=
\begin{cases}
\dfrac{w_\delta'(x)}
{\bigl(1+w_\delta(x)^2\bigr)^{3/2}},
& x\in[0,1]\setminus\mathbb D_\delta,\\[0.5em]
0,
& x\in\mathbb D_\delta.
\end{cases}
\]
Then, using the decomposition \eqref{eq:Dcu} $ |Dc_u|=\sqrt{1+w_\delta^2}\,\mathcal L^1+D^c u $ and the fact that $g=0$ on $\mathbb D_\delta$, we obtain
\begin{align*}
\int_{[0,1]} |g|^2\diff |Dc_u|
&= \int_{[0,1]\setminus\mathbb D_\delta} \frac{|w_\delta'(x)|^2} {\bigl(1+w_\delta(x)^2\bigr)^3}\sqrt{1+w_\delta(x)^2}\diff x
\\
&= \int_0^1 \frac{|w_\delta'(x)|^2}{\bigl(1+w_\delta(x)^2\bigr)^{5/2}}\diff x
<\infty .
\end{align*}
Thus, we conclude $g\in L^2([0,1],|Dc_u|)$ and with the push-forward identity \eqref{eq:sigmapush}
$ g\circ s^{-1}\in L^2(0,L). $

\noindent\textbf{Claim.} For every $x\in[0,1]$, one has
\begin{align}
\theta(x)
=
\frac{\pi}{2}
+
\int_{[0,x]} g\diff |Dc_u|.
\label{eq:thetaf}
\end{align}
\textit{\textbf{Proof of \eqref{eq:thetaf}:}}
Write the connected components of $[0,1]\setminus\mathbb D_\delta$ as $ I_{jk}=(\alpha_{jk},\beta_{jk})$. Since $w_\delta(x)\to+\infty$ as
$x\searrow\alpha_{jk}$ and as $x\nearrow\beta_{jk}$, we have
\[
\lim_{x\searrow\alpha_{jk}}\theta(x)
=
\lim_{x\nearrow\beta_{jk}}\theta(x)
=
\frac{\pi}{2}.
\]
On each $I_{jk}$, the function $w_\delta$ is $C^1$, and hence {on }$I_{jk}$ we obtain
\[
\theta'
=
\frac{w_\delta'}{1+w_\delta^2} \quad \text{on }I_{jk}
\quad
\Rightarrow
\quad 
g\diff |Dc_u|
=
\frac{w_\delta'}{1+w_\delta^2} \diff x
=
\theta'\diff x
\quad\text{on }I_{jk},
\]
since on $I_{jk}$ we have $\diff |Dc_u| = \sqrt{1+w_\delta^2}\diff x$.
Consequently,
\begin{align}
\int_{I_{jk}} g \diff |Dc_u|
&=
\int_{\alpha_{jk}}^{\beta_{jk}}
\frac{w_\delta'(x)}{1+w_\delta(x)^2}\diff x =
\lim_{R\nearrow\beta_{jk}}\theta(R)-\lim_{R'\searrow\alpha_{jk}}\theta(R')
=
0 .
\label{eq:int-g-over-gap-zero}
\end{align}

Now let $x\in[0,1]$. We distinguish two cases.
First assume that $x\in\mathbb D_\delta$. Since $g=0$ on $\mathbb D_\delta$, and since the connected components of $[0,1]\setminus\mathbb D_\delta$ are precisely the intervals $I_{jk}$, we obtain
\[
\int_{[0,x]} g\diff |Dc_u|
=
\sum_{I_{jk}\subset (0,x)}
\int_{I_{jk}} g\diff |Dc_u|.
\]
The series is well-defined because $g\in L^1([0,1],|Dc_u|)$, which follows
from $g\in L^2([0,1],|Dc_u|)$ and $|Dc_u|([0,1])<\infty$. By
\eqref{eq:int-g-over-gap-zero}, every summand vanishes. Hence
\[
\frac{\pi}{2}+\int_{[0,x]} g\diff |Dc_u|=\frac{\pi}{2}=\theta(x).
\]

Now assume that $x\notin\mathbb D_\delta$. Then there is a unique interval $I_{j_0k_0}=(\alpha_{j_0k_0},\beta_{j_0k_0})$ such that $x\in I_{j_0k_0}$. Since $\alpha_{j_0k_0}\in\mathbb D_\delta$, the first case gives
\begin{align*}
\int_{[0,x]} g\diff |Dc_u|
&=
\int_{[0,\alpha_{j_0k_0}]} g\diff |Dc_u|
+
\int_{(\alpha_{j_0k_0},x]} g\diff |Dc_u| \\
&=
0+\int_{\alpha_{j_0k_0}}^x
\frac{w_\delta'(t)}{1+w_\delta(t)^2}\diff t 
=
\theta(x)-\lim_{R\searrow \alpha_{j_0k_0}}\theta(R) \\
&=
\theta(x)-\frac{\pi}{2}.
\end{align*}
This proves \eqref{eq:thetaf}. 

Next, we define the \emph{tangent angle parametrized over length} by
\[
\vartheta\colon[0,L] \to [0,\pi/2], \qquad\vartheta:=\theta\circ s^{-1}.
\]
Since $s^{-1}_\#\mathcal L^1=|Dc_u|$, identity \eqref{eq:thetaf} gives for all $\ell\in[0,L]$
\[
\vartheta(\ell)
=
\frac{\pi}{2}+\int_0^\ell g(s^{-1}(r))\diff r
\quad \Rightarrow \quad
\vartheta\in W^{1,1}((0,L)),
\quad
\vartheta'=g\circ s^{-1}
\quad\text{ a.e. in }(0,L).
\]
Furthermore, since $g=0$ on $\mathbb D_\delta$ we conclude
\begin{align*}
\int_0^L |\vartheta'(\ell)|^2\diff\ell
&=
\int_0^L |g(s^{-1}(\ell))|^2\diff\ell
\overset{\eqref{eq:sigmapush}}=
\int_{[0,1]} |g(x)|^2\diff|Dc_u|\\
&=
\int_0^1
\frac{|w_\delta'(x)|^2}{(1+w_\delta(x)^2)^3}
\sqrt{1+w_\delta(x)^2}\diff x
=
\int_0^1
\frac{|w_\delta'(x)|^2}{(1+w_\delta(x)^2)^{5/2}}\diff x.
\end{align*}
The last integral is finite by Step~\cbm 5 of the construction. Hence, we obtain $\vartheta\in W^{1,2}((0,L))$.
By construction of the arc-length parametrization, it holds for all $\ell \in[0,L]$
\[
\gamma'(\ell)=\tau(s^{-1}( \ell))
=
\bigl(\cos\vartheta(\ell),\sin\vartheta(\ell)\bigr).
\]
The Sobolev chain rule therefore gives
$
\gamma'\in W^{1,2}((0,L);\mathbb R^2)
\Rightarrow  \gamma\in W^{2,2}((0,L);\mathbb R^2).
$
Moreover, for the second derivative it holds
\[
\gamma''(\ell)
=
\vartheta'(\ell)\,(-\sin\vartheta(\ell),\cos\vartheta(\ell))
\qquad\text{for a.e. }\ell\in(0,L).
\]
Since $|\gamma'|\equiv1$, the scalar curvature of the arc-length parametrized curve is
\[
\kappa_\gamma(\ell)=\vartheta'(\ell)=
\begin{cases}
\dfrac{w_\delta'(s^{-1}(\ell))}{(1+w_\delta(s^{-1}(\ell))^2)^{3/2}},
& \ell\in s\left([0,1]\setminus\mathbb D_\delta\right),\\
0,
& \ell\in s(\mathbb D_\delta).
\end{cases} 
\quad\text{for a.e. }\ell\in(0,L).
\]
Therefore, \eqref{eq:kappa-L2} follows from
\[
\int_0^L |\kappa_\gamma|^2(\ell)\diff\ell
=
\int_0^L |\vartheta'(\ell)|^2\diff\ell
=
\int_0^1
\frac{|w_\delta'(x)|^2}{(1+w_\delta(x)^2)^{5/2}}\diff x=\W^a(u).
\]


\medskip
\noindent\cb{3} \textbf{Failure of local $C^2$-regularity.}\\
We prove that the classical curvature of the smooth graph pieces becomes unbounded along sequences of points converging to Cantor points. Let $I_{jk}=(\alpha_{jk},\beta_{jk})$ be one of the removed intervals at level $j$. Its length is $(1-2\delta)\delta^{j-1}$ and  $a_{jk}$  denotes its midpoint. Define the affine rescaling
\begin{align*}
    \rho_{jk}(x)
:= \frac{x-a_{jk}}{(1-2\delta)\delta^{j-1}}+\frac12, 
\qquad x\in I_{jk}.
\end{align*}
Then $\rho_{jk}$ maps $I_{jk}$ bijectively onto $(0,1)$, with
$\rho_{jk}(\alpha_{jk})=0$ and $ \rho_{jk}(\beta_{jk})=1$.

Since $f_\delta$ is constant on every removed interval $I_{jk}$, the function
$u=U+f_\delta$ satisfies on $I_{jk}$
\begin{align*}
u'(x)=U'(x)=w_\delta(x)
= 2^{sj}+2^{sj}\Phi(\rho_{jk}(x)),
\quad 
u''(x)=w_\delta'(x)
=\frac{2^{sj}}{(1-2\delta)\delta^{j-1}}\Phi'(\rho_{jk}(x)).
\end{align*}
Thus $u$ is $C^2$ on each $I_{jk}$, and the classical signed curvature of
the graph of $u$ at points $x\in I_{jk}$ is given by
\begin{align*}
\kappa_u(x)
= \frac{u''(x)}{\bigl(1+u'(x)^2\bigr)^{3/2}}
= \frac{w_\delta'(x)} {\bigl(1+w_\delta(x)^2\bigr)^{3/2}}.
\label{eq:local-graph-curvature}
\end{align*}
This is the standard curvature formula for a Cartesian graph, see, for instance, \cite[(2.5)]{acerbi2017geometric}.

We first present two elementary estimates for $\Phi$. Later we shall choose $y=y_j\searrow0$. Therefore it is enough to consider the range
\[0<y\le\frac14
\qquad \Rightarrow\qquad
1-y\ge\frac34,
\quad
|2y-1|\ge\frac12.
\]
Hence there exists a constant $C_1>0$, depending only on $\beta$, such that
\begin{align*}
0\le \Phi(y)
= C_\beta\left(\frac{1}{y^\beta(1-y)^\beta}-4^\beta\right)
\le C_1 y^{-\beta} \qquad\text{for }0<y\le\frac14 .
\end{align*}
Moreover, there exists a constant $C_2>0$, depending only on $\beta$, such that
\begin{align*}
|\Phi'(y)|
= C_\beta\beta \frac{|2y-1|}{y^{\beta+1}(1-y)^{\beta+1}}
\ge C_2 y^{-\beta-1} \qquad\text{for }0<y\le\frac14 .
\end{align*}
We then conclude
\begin{align}
\big|\kappa_u(\rho_{jk}^{-1}(y))\big|
\succeq \frac{2^{sj}}{2^{3sj}(1-2\delta)\delta^{j-1}} \frac{y^{-\beta-1}}{(1+y^{-\beta})^3}
\succeq \delta^{-j}2^{-2sj} y^{2\beta-1} 
\qquad\text{for }0<y\le\frac14.
\label{eq:kappa-est}
\end{align}


We now show that, near every Cantor point, there are removed intervals of arbitrarily high generation. Fix
\[
\chi\in\mathbb D_\delta .
\]
Since $\chi\in\mathbb D_\delta$, it is never removed in the Cantor construction. Hence, for every $j\in\mathbb N_0$, there exists a closed interval $J_j=J_j(\chi)$ of the $j$-th generation such that
\[
\chi\in J_j,
\qquad
|J_j|=\delta^j, \qquad 
\bigcap_{j=0}^\infty J_j=\{\chi\}.
\]
and the intervals may be chosen nested: $J_{j+1}\subset J_j $. 
For $j\ge1$, let $I_j^*=I_{jk_j }$ be the middle open interval removed from $J_{j-1}$ when passing from the $(j-1)$-st to the $j$-th generation. 
Then $I_j^*\subset J_{j-1}$ and $ |I_j^*|=(1-2\delta)\delta^{j-1}$.
Since $\chi\in J_{j-1}$ and $\operatorname{diam}(J_{j-1})=\delta^{j-1}\to0$, we have
\[
\operatorname{dist}(I_j^*,\chi)\le \delta^{j-1}\to0 .
\]
Consequently, every choice of points $x_j\in I_j^*$ satisfies $x_j\to\chi$ as $j\to\infty$.


Since, by condition \eqref{eq:1dcond2}, we have $2\beta-1<0$, we define
for $j$ sufficiently large
\begin{align}
   y_j:=2^{2sj/(2\beta-1)}. \label{eq:def_yj} 
\end{align}
Then $y_j\to0$, and hence $y_j\in(0,\frac14)$ for all sufficiently large $j$. Since $ \rho_{jk_j }:I_j^*\to(0,1)$ is affine and bijective, we may choose $x_j\in I_j^*$ such that
\begin{align}
\rho_{jk_j }(x_j)=y_j\quad \Rightarrow \quad x_j\underset{j\to\infty}\longrightarrow\chi  \label{eq:def_xj}    
\end{align}
since $I_j^*\to\chi$ in the sense that $\operatorname{dist}(I_j^*,\chi)\to0$.
Moreover, by the choice of $y_j$ and by using \eqref{eq:kappa-est}, we therefore obtain
\[
|\kappa_u(x_j)|
\ge
C_4\,\delta^{-j}2^{-2sj}y_j^{2\beta-1}
=
C_4\,\delta^{-j}
\to+\infty
\qquad\text{as }j\to\infty,
\]
because $0<\delta<1$.
Thus the classical curvature of the smooth graph pieces is unbounded along the sequence $x_j\to\chi$.


Since the classical curvature of the smooth graph pieces is unbounded
along $p_j:=(x_j,u(x_j))\to p_*:=(\chi,u(\chi))$, the graph
cannot admit a regular local $C^2$-parametrization at $p_*$. Otherwise,
the curvature of such a parametrization would be continuous and hence
locally bounded, contradicting the invariance of curvature under
regular $C^2$-reparametrization on the smooth graph pieces.

 Since $\chi\in\mathbb D_\delta$ was arbitrary, the graph fails to be locally $C^2$ at every point of $c_u(\mathbb D_\delta)$, with the one-sided interpretation at its two endpoints. In particular, $c_u([0,1])$ is not a regular embedded $C^2$-curve with boundary.
\end{proof}

\begin{remark}[A variant with a $C^2$-regular curve]
\label{rem:C2-variant}
The failure of $C^2$-regularity in Theorem~\ref{thm:c1curv} can be
avoided by changing the parameter regime. Choose
\[
    \beta\in\left(\frac12,1\right),
    \qquad
    s>1,
    \qquad
    2^{-2s}<\delta<2^{-s-1}.
\]
Since $s>1$, the interval for $\delta$ is nonempty and
\[
    2^{1-3s}<2^{-2s}<\delta.
\]
Thus the summability conditions used in the construction remain
satisfied. Moreover, $\beta<1$ guarantees the integrability of the
profile, while $\beta>\frac13$ guarantees the finiteness of the local
Willmore integrals. Consequently, repeating the construction of
Theorem~\ref{thm:finit_will} with these parameters yields a function
\[
    u\in BV((0,1))\cap C^0([0,1]),
    \qquad
    D^c u\neq0,
    \qquad
    \overline{\W}(u)<\infty.
\]

We claim that the graph of $u$ is now a regular embedded $C^2$-curve. Let $I_{jk}$ be a removed interval of generation $j$ and again write $  y=\rho_{jk}(x)$. Then
\begin{align}
\left|
\kappa_u\bigl(\rho_{jk}^{-1}(y)\bigr)
\right|
&\leq
\frac{\delta}{1-2\delta}\,
\delta^{-j}2^{-2sj}
\frac{|\Phi'(y)|}
     {\bigl(1+\Phi(y)\bigr)^3}.
\label{eq:C2-curvature-upper}
\end{align}

For $0<y\leq\frac14$, the estimates used in
\eqref{eq:kappa-est} can be reversed in the appropriate places:
\begin{align*}
\Phi(y)
= &\ C_\beta\left(\frac{1}{y^\beta(1-y)^\beta}-4^\beta\right)
 \succeq y^{-\beta} &&\text{for }0<y\le\frac14 ,\\
|\Phi'(y)|
=&\ C_\beta\beta \frac{|2y-1|}{y^{\beta+1}(1-y)^{\beta+1}}
\preceq y^{-\beta-1}&&\text{for }0<y\le\frac14 .
\end{align*}
We then conclude
\begin{align*}
\big|\kappa_u(\rho_{jk}^{-1}(y))\big|
\preceq \frac{2^{sj}}{2^{3sj}(1-2\delta)\delta^{j-1}} \frac{y^{-\beta-1}}{(1+y^{-\beta})^3}
\preceq \delta^{-j}2^{-2sj} y^{2\beta-1} 
\qquad\text{for }0<y\le\frac14.
\end{align*}
By the symmetry of $\Phi$, the analogous estimate holds near $y=1$:
\[
    \frac{|\Phi'(y)|}
         {\bigl(1+\Phi(y)\bigr)^3}
    \leq C(1-y)^{2\beta-1}
    \qquad
    \text{for }\frac34\leq y<1.
\]
Since $2\beta-1>0$, we therefore have
\[
    \kappa_u\bigl(\rho_{jk}^{-1}(y)\bigr)
    \longrightarrow0
\]
as $y\searrow0$ or $y\nearrow1$. Thus the curvature on each removed interval extends continuously to its endpoints by assigning the value zero there. In particular, $h_\beta$ is bounded on $[0,1]$. It follows from \eqref{eq:C2-curvature-upper} that
\[
    \sup_{x\in I_{jk}}|\kappa_u(x)|
    \leq
    C\delta^{-j}2^{-2sj}
    =
    C\left(\frac{2^{-2s}}{\delta}\right)^j.
\]
Since $\delta>2^{-2s}$, the right-hand side converges to zero as $j\to\infty$.

 Hence the scalar curvature on $[0,1]\setminus\mathbb D_\delta$ extends continuously to $[0,1]$. Let $s:[0,1]\to[0,L]$ and $\gamma=c_u\circ s^{-1}$ be the arc-length parametrization from Theorem~\ref{thm:c1curv}. The extended scalar curvature, pulled back by $s^{-1}$, is continuous on $[0,L]$. The weak Frenet identity proved there, $\gamma''=\kappa_\gamma\nu_\gamma$ almost everywhere, therefore has a continuous right-hand side. Hence $\gamma'\in C^1([0,L];\mathbb R^2)$ and $\gamma\in C^2([0,L];\mathbb R^2)$. Since $|\gamma'|=1$, the image $c_u([0,1])$ is a regular one-dimensional $C^2$-submanifold of $\mathbb R^2$ with boundary.
\end{remark}


\begin{remark}[Interpretation of Menne's theorem in the present example]
\label{rem:menne}
The example in Thm.~\ref{thm:c1curv} should not be read as contradicting Menne's second-order rectifiability theorem in \cite[Thm.~3.6]{menne2013second}. On the contrary, it illustrates the precise measure-theoretic meaning of that theorem.

Let $V_u$ be the multiplicity-one integral varifold associated with $\Gamma_u=\gamma([0,L])$. Since $\gamma\in W^{2,2}(0,L;\mathbb R^2)$ is arc-length parametrized, $V_u$ has locally bounded first variation. More precisely, in the open set
\[
    U_0:=\mathbb R^2\setminus\{\gamma(0),\gamma(L)\}
\]
its first variation is represented by integration,
\[
    \delta V_u(X)
    =
    -\int_{\Gamma_u} H_{V_u}\cdot X\,d\|V_u\|,
    \quad
    H_{V_u}(\gamma(\ell))=\gamma''(\ell),
\quad
    \int_{\Gamma_u}|H_{V_u}|^2\,d\|V_u\|
    =
    \int_0^L|\kappa_\gamma(\ell)|^2\,d\ell
    <\infty .
\]
If the endpoints are included in the ambient open set, the first variation has in addition the two usual boundary atoms. In either case $\|\delta V_u\|$ is a Radon measure, and hence Menne's second-order rectifiability theorem applies.

Consequently, there are countably many $C^2$-curves $\Sigma_i\subset\mathbb R^2$ such that
\[
    \|V_u\|\left(
        \Gamma_u\setminus \bigcup_{i=1}^{\infty}\Sigma_i
    \right)=0 .
\]
Moreover, on these $C^2$-pieces the classical curvature agrees $\|V_u\|$-almost everywhere with the generalized mean curvature.

The point is that this conclusion is only a measure-theoretic covering statement. It does not imply that $\Gamma_u$ itself is a $C^2$-curve, nor that $\Gamma_u$ can be decomposed into countably many $C^2$-subarcs lying inside the graph. In the present construction the Cantor trace has positive $\|V_u\|$-measure, while $\Gamma_u$ is not locally a regular $C^2$-curve at any Cantor point. Hence some of the covering curves $\Sigma_i$ must pass through uncountably many points of the Cantor trace, but they cannot coincide with $\Gamma_u$ in any neighbourhood of such a point. They may therefore  cover Cantor points in the sense of Menne's theorem while leaving the graph between those points.

Thus the example shows that countable $C^2$-rectifiability of an integral varifold with $L^2$-mean curvature is much weaker than a global, or even locally compatible, $C^2$-curve structure of its support.
\end{remark}


\section{General \texorpdfstring{$p>1$}{p>1} version of the Cantor construction}
In this section we show that the mechanism behind Thm.~\ref{thm:finit_will} is not specific to the Willmore case $p=2$, but works for every exponent $p>1$, provided the weight $\psi$ decays like $t^{-\alpha}$ with $\alpha>2p-1$. This extends the examples of \cite[Remark~3.2~(iv)]{maso2009higher}, which cover only exponents $p$ close to~$1$, to the full range $p>1$. Since the weight associated with the $p$-curvature energy has decay exponent $\alpha=3p-1>2p-1$, the relaxed $L^p$-curvature energies $\overline{\W}_p$ are included for every $p>1$ (Corollary~\ref{cor:Lp-curvature-cantor}). As before, the construction admits a geometric interpretation: the associated curve has an arc-length parametrization of class $C^1\cap W^{2,p}$, and a variant of the singular profile even yields a $C^2$-regular image (Remark~\ref{rem:C1W2pcurves}). As a further application, we obtain a surface of revolution with finite Willmore energy whose profile function carries a nontrivial Cantor part (Corollary~\ref{cor:rotational-willmore-cantor}). 
\begin{thm}[Cantor parts in $X_\psi^p$ for every $p>1$]
\label{thm:cantor_part_Xp}
Let $p>1$, and let $ \psi:\mathbb R\to (0,+\infty)$  be a bounded Borel function satisfying \eqref{eq:psibed}. 
Assume moreover that there exist constants $C>0$ and
\[
    \alpha>2p-1
\quad \text{ such that } \quad
    \psi(t)\le \frac{C}{t^\alpha}
    \qquad
    \text{for all }t\ge1 .
\]
Then, for every bounded open interval $I=(a,b)$, there exists a function
\[
    u\in X_\psi^p(I)
\]
such that its distributional derivative has a nontrivial Cantor part. 
\end{thm}

\begin{proof} 
According to the definition in \cite{maso2009higher}, for $p>1$ the class $X_\psi^p$ is described in terms of the absolutely continuous density 
$(u')^a=\frac{dDu}{d\mathcal L^1}$ by requiring
\[
\Psi_p\bigl((u')^a\bigr)\in W^{1,p},
\]
where $\Psi_p$ is defined in \eqref{eq:psi_p}. Moreover, the singular part $D^s u$ has to be concentrated on the set where the chosen continuous representative of $(u')^a$ takes the values $\pm\infty$. In the construction below only the value $+\infty$ occurs.

It is enough to consider $I=(0,1)$. The construction on an arbitrary bounded open interval $I=(a,b)$ is obtained by applying the same argument after an affine change of variables in the domain. All estimates remain unchanged up to constants depending only on
$|I|$.

In this proof we only highlight the changes that have to be made in the proof of Thm.~\ref{thm:finit_will}. In Step~\cbm 1, we replace the definitions of $M$ and $\Psi_2$ by 
\[
    M_p:=\int_{-\infty}^{+\infty}\psi(t)^{1/p}\diff t<+\infty,
    \qquad
    \Psi_p(t):=\int_{-\infty}^{t}\psi(s)^{1/p}\diff s .
\]
The finiteness of $M_p$ is part of assumption \eqref{eq:psibed}. The additional decay estimate is used only to control the positive tail and to obtain the quantitative bounds required for the Cantor-profile construction.  Indeed, since $\alpha>2p-1>p$, one has
\[
\int_1^\infty \psi(t)^{1/p} \diff t
\le
C^{1/p}\int_1^\infty t^{-\alpha/p}\diff t
<\infty .
\]
Then $\Psi_p$ extends continuously to $[-\infty,+\infty]$, with $\Psi_p(-\infty)=0$ and $\Psi_p(+\infty)=M_p$. The Cantor construction itself  in Step~\cbm 2 remains unchanged. In Step~\cbm 3, we replace the condition on $\beta$ in \eqref{eq:1dcond2} by 
\begin{align}
       \beta\in
    \left(\frac{p-1}{\alpha-p},1\right).
    \label{eq:1dcond2p}\tag{C2p}
\end{align}
This interval is nonempty precisely because $\alpha>2p-1$. The definition of function $\Phi$ also remains completely unchanged. The condition \eqref{eq:choice_s} on $s>1$ is replaced by
\begin{align}
   s>\frac{p}{\alpha-2p+1} \quad \Longleftrightarrow\quad  2^{\frac{1-s(\alpha-p)}{p-1}}<2^{-s-1},
\label{eq:choice_sp}\tag{C3p}
\end{align}
which is later important for the choice of $\delta$. The definitions of $\Phi_{jk}$ and $w_\delta$ are not altered and the proof that $w_\delta\in L^1(0,1)$ remains the same:
\begin{align*}
    \int_0^1 w_\delta(x)\diff x
    = 2^{s+1}\frac{1-2\delta}{1-2^{s+1}\delta}<\infty
\end{align*}
provided the unchanged condition \eqref{eq:condition} is valid:
\begin{align}
    \delta<2^{-s-1}
    \quad\Longleftrightarrow\quad
    2^{s+1}\delta<1.
    \tag{C4}
\end{align}

Moreover, $w_\delta$ is continuous as a function with values in $[0,+\infty]$. At endpoints of removed intervals this follows from the divergence of $\Phi$ at $0$ and $1$. If $\chi\in\mathbb D_\delta$ is not such an endpoint and $x_n\to\chi$ with $x_n\notin\mathbb D_\delta$, then $x_n\in I_{j_n k_n}$ with $j_n\to\infty$. Hence
\[
w_\delta(x_n)\ge 2^{s j_n}\to+\infty.
\]
Since $w_\delta=+\infty$ on $\mathbb D_\delta$, this proves continuity at
$\chi$.


In Steps~\cbm 4 and~\cbm 5, we argue as before, replacing $\Psi_2$ by $\Psi_p$ throughout. On each connected component of $(0,1)\setminus\mathbb D_\delta$, the function $\Psi_p\circ w_\delta$ is smooth. We therefore define the natural candidate for its weak derivative on $(0,1)$ by
\[
    G_\delta(x):=
    \begin{cases}
        (\Psi_p\circ w_\delta)'(x),
        & x\in[0,1]\setminus\mathbb D_\delta,\\
        0,
        & x\in\mathbb D_\delta.
    \end{cases}
\]
Here we have to prove that 
\[
    \Psi_p\circ w_\delta\in W^{1,p}((0,1)), \quad
(\Psi_p\circ w_\delta)'=G_\delta
\quad\text{a.e. in }(0,1).
\]
Since $w_\delta$ is continuous with values in $[0,+\infty]$ and $\Psi_p$ is continuous on $[0,+\infty]$, we have \[ \Psi_p\circ w_\delta\in C^0([0,1])\subset L^p((0,1)). \]
Furthermore,  $\psi$ is bounded and Borel measurable. Therefore, $\Psi_p$ is locally absolutely continuous on $\mathbb R$, with
\[
\Psi_p'(t)=\psi(t)^{1/p}
\qquad\text{for a.e. }t\in\mathbb R.
\]
Moreover, $w_\delta\in C^1_{\mathrm{loc}}(I_{jk})$. Hence $\Psi_p\circ w_\delta\in W^{1,1}_{\mathrm{loc}}(I_{jk})$ is locally absolutely continuous on $I_{jk}$, and on each interval $I_{jk}$ the one-dimensional chain rule gives 
\begin{align*}
    (\Psi_p\circ w_\delta)'(x) &= (\Psi_p\circ \Phi_{jk})'(x)=\Psi_p'\big(\Phi_{jk}(x)\big) \cdot \Phi'_{jk}(x)\\
     &= \frac{{2^{sj}}}{(1-2\delta)\delta^{j-1}}\cdot\psi^{1/p}\big(\Phi_{jk}(x) \big)\cdot \Phi' \left(  \frac{x-a_{jk}}{(1-2\delta)\delta^{j-1}}  + \frac{1}{2}\right) 
\end{align*}
 instead of \eqref{eq:psi_abl}. 
Since $\Phi_{jk}(x) \ge 2^{sj}\ge1$, the decay assumption on $\psi$ gives
\[
    \psi\big(\Phi_{jk}(x) \big)
    \le
    C\big(\Phi_{jk}(x) \big)^{-\alpha}.
\]
Using the explicit form of $w_\delta$ on $I_{jk}$, the decay assumption on $\psi$, and the change of variables $y=\rho_{jk}(x)$, we obtain
\[
\begin{aligned}
    \int_{I_{jk}} \big|(\Psi_p\circ w_\delta)'(x)\big|^p\diff x
    &\preceq  \frac{{2^{sj(p-\alpha)}}}{(1-2\delta)^p\delta^{(j-1)p}}\cdot
    \int_{I_{jk}} \frac{\left| \Phi' \left(  \frac{x-a_{jk}}{(1-2\delta)\delta^{j-1}}  + \frac{1}{2}\right)\right|^p}{\left|1+ \Phi \left(  \frac{x-a_{jk}}{(1-2\delta)\delta^{j-1}}  + \frac{1}{2}\right)\right|^\alpha}\diff x                                      \\
      &\preceq   \frac{{2^{sj(p-\alpha)}}}{(1-2\delta)^{p-1}\delta^{(j-1)(p-1)}}\cdot
    \int_{0}^1 \frac{\left| \Phi' \left( y\right)\right|^p}{\left|1+ \Phi \left( y\right)\right|^\alpha}\diff y                                      \\
    &\underset{\eqref{eq:condition}}{\mathclap{\overset{2\delta<2^{-s}}\preceq} } \quad 
   \frac{2^{sj(p-\alpha)}}{\delta^{j(p-1)}}
    \int_0^1 
    y^{-p(\beta+1)}y^{\alpha\beta} \diff y.
\end{aligned}
\]
The last integral is integrable if and only if $-p(\beta+1)+\alpha\beta>-1$.
This is exactly our choice of $\beta$ in \eqref{eq:1dcond2p}. 
Consequently,
\[
    \int_{I_{jk}} \big|(\Psi_p\circ w_\delta)'(x)\big|^p\diff x
    \preceq
    (2^{s(p-\alpha)}\delta^{1-p})^j.
\]
Summing over all removed intervals, we obtain
\begin{align*}
    \int_{[0,1]\setminus\mathbb D_\delta}
    \big|(\Psi_p\circ w_\delta)'(x)\big|^p\diff x
    &=  \sum_{j=1}^\infty \sum_{k=1}^{2^{j-1}} \int_{I_{jk}} \big|(\Psi_p\circ w_\delta)'(x)\big|^p\diff x   \preceq \sum_{j=1}^\infty \left( 2^{1+s(p-\alpha)}\delta^{1-p} \right)^j .
\end{align*}
Therefore, the last geometric series converges provided that we replace \eqref{eq:condition2} by 
\begin{align}
2^{1+s(p-\alpha)}\delta^{1-p}<1
\quad\Longleftrightarrow\quad
\delta> 2^{\frac{1-s(\alpha-p)}{p-1}}.
\label{eq:condition2p}\tag{C5p}
\end{align}
By the choice of $s$ in \eqref{eq:choice_sp}, this condition is compatible with the upper bound $\delta<2^{-s-1}$. Hence we may choose
\[
2^{\frac{1-s(\alpha-p)}{p-1}}
<
\delta
<
2^{-s-1}.
\]
Thus, $ (\Psi_p\circ w_\delta)'\in L^p([0,1]\setminus\mathbb D_\delta). $ 

It remains to characterize \(G_\delta\) as the weak derivative of \(\Psi_p\circ w_\delta\). Since \(\Psi_p\circ w_\delta\) is locally absolutely continuous on every removed interval $I_{jk}$ and tends to \(M_p\) at both endpoints of each such interval, the same argument as in Step~\cbm 4 of Thm.~\ref{thm:finit_will} gives
\[
    (\Psi_p\circ w_\delta)(x)
    =
    M_p+\int_0^x G_\delta(\xi)\,\diff\xi
    \qquad\text{for all }x\in[0,1].
\]
Together with \(G_\delta\in L^p((0,1))\) similar to Step~\cbm 5 of Thm.~\ref{thm:finit_will}, this yields
\[
    \Psi_p\circ w_\delta\in W^{1,p}((0,1)),
    \qquad
    (\Psi_p\circ w_\delta)'=G_\delta
    \quad\text{a.e.}
\]

As in Step~\cbm{6} of the proof of Thm.~\ref{thm:finit_will}, we define for all $x\in[0,1]$
\[
   u(x):=U(x)+f_\delta(x) \quad \text{ with }\quad U(x):=\int_0^x w_\delta(t) \diff t
\]
where $f_\delta$ is the Cantor function associated with $\mathbb D_\delta$. Since $w_\delta\in L^1((0,1))$, we have $U\in W^{1,1}((0,1))$. Moreover, $f_\delta$ is continuous and nondecreasing. Hence, it follows that $ u\in BV((0,1))\cap C^0([0,1])$.  
In particular,
\[
(u')^a=w_\delta
\qquad\mathcal L^1\text{-a.e.},
\quad \text{
and
}\quad
D^s u=D^c u=D^c f_\delta.
\]
All pointwise statements about the blow-up set of \((u')^a\) are understood with respect to the chosen continuous representative \(w_\delta\). Since $w_\delta(x)=+\infty \ \Longleftrightarrow\  x\in\mathbb D_\delta$, the chosen continuous representative of $(u')^a$ satisfies
\[
    Z^+[(u')^a]=\mathbb D_\delta\cap(0,1),
    \qquad
    Z^-[(u')^a]=\emptyset.
\]
The singular measure $ D^s u$ is positive and supported on $\mathbb D_\delta$. Hence its positive part is concentrated on $Z^+[(u')^a]$, while its negative part is zero and therefore
trivially concentrated on $Z^-[(u')^a]$.

Together with
$ \Psi_p\bigl((u')^a\bigr)
    =
    \Psi_p\circ w_\delta
    \in W^{1,p}((0,1))
$ it follows that $
    u\in X_\psi^p(0,1).
$
Additionally, $Du$ has a nontrivial Cantor part. The analogue of Step~\cbm{7} of Thm.~\ref{thm:finit_will} is not needed here, since the corresponding structural result is already contained in \cite[Thm.~3.4.]{maso2009higher}. This concludes the proof.
\end{proof}

\begin{corollary}[Cantor parts for \texorpdfstring{$L^p$}{Lp}-curvature energies]
\label{cor:Lp-curvature-cantor}
Let $p>1$.  Then there exists a function $ u\in BV((a,b))\cap C^0([a,b])$ such that
\[
    D^c u\neq 0
    \qquad\text{and}\qquad
    \overline{\W}_p(u)<\infty .
\]
In particular, finite relaxed $L^p$-curvature energy does not imply that $u\in SBV((a,b))$. 
\end{corollary}

\begin{proof} For a classical $u$ it holds
\[
    \W_p(u)
=
    \int_a^b \psi_p(u')|u''|^p \diff x
\quad\text{ 
with }\quad
    \psi_p(t):=(1+t^2)^{\frac{1-3p}{2}} .
\]
The function $\psi_p$ is bounded, positive, and satisfies the structural
condition \eqref{eq:psibed}. Moreover,
\[
    \psi_p(t)
    \le C t^{-(3p-1)}
    \qquad\text{for }t\ge1 .
\]
Therefore the decay exponent is $\alpha=3p-1>2p-1$.

Thm.~\ref{thm:cantor_part_Xp}, applied with
$\psi=\psi_p$, yields a function $u\in X_{\psi_p}^p((a,b))$ with nontrivial Cantor part $D^c u\neq0$. By the relaxation theorem for $\mathcal F_p$, the membership $u\in X_{\psi_p}^p((a,b))$ implies $  \overline{\mathcal F}_p(u)<\infty $. Since, for smooth functions, $ \W_p(u)\le \mathcal F_p(u)$, the same recovery sequence also gives $ \overline{\W}_p(u)<\infty . $ This proves the claim.
\end{proof}
\begin{remark}[Geometric interpretation]
\label{rem:C1W2pcurves}
For the specific curvature weight \(\psi_p(t)=(1+t^2)^{(1-3p)/2}\), the function constructed in Corollary~\ref{cor:Lp-curvature-cantor} admits the same geometric interpretation as in Theorem~\ref{thm:c1curv}. The Cartesian curve $ c_u:[0,1]\to\mathbb R^2$, associated with the function $u$ constructed above admits an arc-length parametrization $\gamma:[0,L]\to\mathbb R^2 $ such that $\gamma\in C^1([0,L];\mathbb R^2)\cap W^{2,p}((0,L);\mathbb R^2)$. 
Moreover, the scalar curvature of the arc-length parametrized curve belongs
to $L^p((0,L))$:
\[
   \overline{\W}_p(u)\le  \int_0^L |\kappa_\gamma(\ell)|^p \diff \ell<\infty .
\]

Furthermore, as in Remark~\ref{rem:C2-variant}, one may obtain a $C^2$-regular image by a different choice of the singular profile. Indeed,  the decay exponent is $ \alpha=3p-1 $. The condition \eqref{eq:1dcond2p} in Thm.~\ref{thm:cantor_part_Xp} requires $\beta\in   \left(\frac{p-1}{\alpha-p},1\right)$.
Since now
\[  \frac{p-1}{\alpha-p}  =
    \frac{p-1}{2p-1}  <   \frac12 ,\]
we may choose instead
\[ \beta\in\left(\frac12,1\right). \]
If, in addition, the choice of $\delta$ is strengthened to
\[  2^{-2s}<\delta<2^{-s-1},  \qquad s>1, \]
then the previous admissibility condition remains satisfied, because
\[ 2^{\frac{1-s(\alpha-p)}{p-1}}  = 2^{-2s-\frac{s-1}{p-1}} <  2^{-2s}.\]
With these choices, the curvature on the smooth graph pieces extends continuously across the Cantor set by setting it equal to zero there. Consequently the graph $c_u([0,1])$ is a one-dimensional $C^2$-submanifold of $\mathbb R^2$ with boundary.
\end{remark}

\begin{corollary}[Cantor parts for Willmore surfaces of revolution]
\label{cor:rotational-willmore-cantor}
There exists \(r\in BV((a,b))\cap C^0([a,b])\), \(r\ge r_0>0\), with \(D^c r\neq0\), such that the surface of revolution generated by the arc-length parametrized graph of \(r\) 
\[
    \Sigma_r
    :=
    \bigl\{(x,r(x)\cos\varphi,r(x)\sin\varphi)
    :x\in(a,b),\ \varphi\in[0,2\pi)\bigr\},
\]
has finite Willmore energy.
\end{corollary}

\begin{proof}
Let $u$ be the function obtained from the construction in the case $p=2$, using the \(C^2\)-variant described in Remark~\ref{rem:C2-variant}, so that the graph $c_u([a,b])$ is a one-dimensional $C^2$-submanifold of $\mathbb R^2$. Moreover, $D^c u\neq0 $. Choose $r_0>0$ and set $r:=r_0+u>0$. Since adding a constant does not change the distributional derivative, we have $D^c r=D^c u\neq0 $. Also, $c_r([a,b])$ is just the vertical translate of $c_u([a,b])$, and hence it is again a one-dimensional $C^2$-submanifold of $\mathbb R^2$.

The arc-length parametrization of $c_r([a,b])$ is denoted by
\[ \gamma:[0,L]\to\mathbb R^2, 
    \qquad
    \gamma(\ell)=(x(\ell),\rho(\ell)).
\]
Since the generating curve is $C^2$, the map
\[
    F:[0,L]\times[0,2\pi]\to\mathbb R^3,
    \qquad
    F(\ell,\varphi)
    :=
    \bigl(
        x(\ell),
        \rho(\ell)\cos\varphi,
        \rho(\ell)\sin\varphi
    \bigr)
\]
is $C^2$. Therefore $ |\partial_\ell F|=1,\ |F_\varphi|=\rho\ge r_0,\ F_\ell\cdot F_\varphi=0$, $F$ is a regular $C^2$-parametrization of the surface of revolution. The principal curvatures are, up to the choice of orientation,
\[
    k_1(\ell,\varphi)=\kappa_\gamma(\ell),
    \qquad
    k_2(\ell,\varphi)=\frac{x'(\ell)}{\rho(\ell)}.
\]
Indeed, $k_1$ is the curvature of the meridian curve, while $k_2$ is the rotational curvature. Therefore
\[
    \int_{\Sigma_r} k_1^2 \diff A
    =
    2\pi
    \int_0^L
    \rho(\ell)\,|\kappa_\gamma(\ell)|^2 \diff\ell
    \le
    2\pi \|\rho\|_{L^\infty((0,L))}
    \int_0^L |\kappa_\gamma(\ell)|^2 \diff\ell
    <\infty .
\]
Here we used that $\rho$ is continuous on the compact interval $[0,L]$ and that $\kappa_\gamma\in L^2((0,L))$.

For the rotational principal curvature we obtain
\[     \int_{\Sigma_r} k_2^2 \diff A
    =     2\pi \int_0^L \frac{x'(\ell)^2}{\rho(\ell)} \diff\ell 
    \le   \frac{2\pi}{r_0}\int_0^L x'(\ell)^2 \diff\ell 
    \le   \frac{2\pi L}{r_0} 
    <\infty .
\]
Since the mean curvature $H$ is either $k_1+k_2$ or $\frac12(k_1+k_2)$, depending on convention, we have in either case $|H|^2\le C\,(k_1^2+k_2^2) $ with a universal constant $C>0$. Hence $\W(\Sigma_r)<\infty$.
\end{proof}

\textbf{Acknowledgements.}
Part of this work was carried out during the second author's doctoral studies under the supervision of the first author. The second author would also like to thank Matteo Novaga and Matthias Röger for their thorough evaluation of the doctoral thesis and for their insightful comments and suggestions.

\section{Appendix}

\begin{lemma}\label{lem:Cantor}
Let $\mathbb D_\delta$ and $(f_\ell)_{\ell\in\mathbb N}$ be defined as in Step \cbm 2
of Thm.~\ref{thm:finit_will}. Then the following assertions hold.

There exists a function $f_\delta\in C^0([0,1])$ with $f_\delta(0)=0, f_\delta(1)=1$ such that
\[
f_\ell\to f_\delta
\qquad\text{uniformly on }[0,1].
\]
and moreover
\[
Df_\delta= D^c f_\delta, \qquad
|D^c f_\delta|((0,1))=1.
\]
where $D^c f_\delta$ is a nonnegative Cantor measure supported on $\mathbb D_\delta$.
\end{lemma}

\begin{proof}
Let $m\geq\ell$. By \eqref{eq:f_ell_formula}, the functions $f_m$ and
$f_\ell$ coincide  on every interval $I_{jk}$ with $j\leq\ell$. Let
$J_{i\ell}=[a_{i\ell},b_{i\ell}]$ be a connected component of
$C_\ell$. Then $f_m$ and $f_\ell$ have the same endpoint values on
$J_{i\ell}$, and
\[
    f_\ell(b_{i\ell})-f_\ell(a_{i\ell})=2^{-\ell}.
\]
Since both functions are nondecreasing, it follows that, for every
$x\in J_{i\ell}$,
\[
    \big|f_m(x)-f_\ell(x)\big|\leq 2^{-\ell}.
\]
Together with the equality on the removed intervals of levels
$j\leq\ell$, this gives
\[
    \|f_m-f_\ell\|_{C^0([0,1])}\leq 2^{-\ell}.
\]
Hence $(f_\ell)$ is Cauchy in $C^0([0,1])$ and converges uniformly to some
$f_\delta\in C^0([0,1])$. Since every $f_\ell$ is nondecreasing and
satisfies $f_\ell(0)=0$, $f_\ell(1)=1$, the same is true for
$f_\delta$.

It remains to characterize the distributional derivative of $f_\delta$. Each $f_\ell$ is nondecreasing, and therefore the uniform limit $f_\delta$ is nondecreasing as well.
Hence $f_\delta\in BV((0,1))$, and $Df_\delta$ is a finite nonnegative Radon measure.

On every removed interval $I_{jk}$, the functions $f_\ell$ are eventually constant.
Passing to the uniform limit, $f_\delta$ is constant on each $I_{jk}$. Therefore
\[
f_\delta'=0
\qquad\text{a.e. on }[0,1]\setminus\mathbb D_\delta.
\]
Since $\mathcal L^1(\mathbb D_\delta)=0$, this implies
$
D^af_\delta=0.
$
Furthermore, $f_\delta$ is continuous, and hence it has no jump part:
$
D^j f_\delta=0.
$
Since
\[
|Df_\delta|((0,1))
=
f_\delta(1)-f_\delta(0)
=
1,
\]
the derivative is nontrivial. Therefore the whole derivative is its Cantor part:
\[
Df_\delta=D^c f_\delta.
\]
Finally, since $f_\delta$ is constant on every connected component of $[0,1]\setminus\mathbb D_\delta$, the measure $D^c f_\delta$ is supported on $\mathbb D_\delta$.
\end{proof}

            \bibliographystyle{alphaurl}
            \bibliography{lib}


            \end{document}